\documentclass{amsart}
\usepackage{amsfonts,amscd}

\newtheorem{theorem}{Theorem}[section]
\newtheorem{lemma}[theorem]{Lemma}
\newtheorem{corollary}[theorem]{Corollary}
\newtheorem{proposition}[theorem]{Proposition}
\theoremstyle{remark}

\theoremstyle{definition}

\numberwithin{equation}{section}
\makeatother

\DeclareMathOperator{\Cdb}{{\mathbb C}}
\DeclareMathOperator{\Rdb}{{\mathbb R}}
\DeclareMathOperator{\Ddb}{{\mathbb D}}

\DeclareMathOperator{\Ndb}{{\mathbb N}}

\begin{document}

\title[Operator algebras with cai III]{Operator algebras with contractive approximate identities, III}

\date{\today}
\thanks{The first author was supported by a grant from the NSF.   
The second author is grateful for  support from UK research council
grant  EP/K019546/1}

\author{David P. Blecher}
\address{Department of Mathematics, University of Houston, Houston, TX
77204-3008}
\email[David P. Blecher]{dblecher@math.uh.edu}

\author{Charles John Read}
\address{Department of Pure Mathematics,
University of Leeds,
Leeds LS2 9JT,
England}
 \email[Charles John Read]{read@maths.leeds.ac.uk}

\begin{abstract}  We continue our study of operator algebras with contractive approximate identities
(cais).  In earlier papers we have introduced and studied a new 
notion of positivity in operator algebras,  with an eye to
extending certain $C^*$-algebraic results and theories to more 
general algebras.  
In the first part of this paper we do a more systematic  development of 
this positivity and its associated ordering, proving many 
foundational facts.  In much of this it is not necessary that the algebra have an approximate identity.
In the second part we study 
interesting examples of operator algebras with cais, 
which in particular answer questions raised in previous papers in this series.  Indeed the present work solves
almost all of the open questions raised in these papers. Many of our results apply immediately to function algebras, but we will not take the time
to point these out, although most of these applications seem new.    
  \end{abstract}

\maketitle

\section{Introduction} 

An {\em operator algebra} is a closed subalgebra of $B(H)$, for a
Hilbert space $H$.  We are mostly
interested in operator algebras with  contractive approximate
identities (cai's).  We also call these {\em approximately unital}
operator algebras.    
 
In earlier papers \cite{BRI,BRII,Read} we introduced and studied a new notion of positivity in operator algebras, with an eye to
extending certain $C^*$-algebraic results and theories to more 
general algebras.  We are also simultaneously developing such extensions (see also e.g.\ \cite{Bnew,BNI,BNII}).  With the same goal in mind, in the first part (Sections 2--6) of the present paper, we undertake a systematic 
study of foundational aspects of this positivity, and of the associated ordering.   In particular, a central role is played by the set ${\mathfrak F}_A = \{ a \in A : 
\Vert 1 - a \Vert \leq 1 \}$ (here $1$ is the identity of the unitization if $A$ is nonunital),  and the cones 
$${\mathfrak c}_A = 
\Rdb_+ \, {\mathfrak F}_A, \; \; \text{and} \; \; {\mathfrak r}_A = \overline{{\mathfrak c}_A} = \{ a \in A : a + a^* \geq 0 \}.$$
Elements of these sets, and their roots, play the role in many situations of positive elements
in a $C^*$-algebra.  
In Section
2 we study general properties of these cones and the related real state space.  Section 3 is
a collection of  results on positivity, some of which are used elsewhere  in this 
paper, and some in forthcoming work.  In  Section 4 we study  `strictly positive' elements, a topic that is quite important for $C^*$-algebras.   In Section 5 
we solve a problem from \cite{BRI} concerning when a right 
ideal $xA$ is already closed, and use this to 
characterize algebraically finitely generated r-ideals  in operator algebras.   Section 6
presents versions of  our previous Urysohn lemma for operator algebras (see e.g.\ \cite{BRI, BNII}), but 
now insisting that the `interpolating element'
is `positive'  in our new sense, 
and is as close  as one would wish to being positive in the usual sense.  This
solves the problems raised at the end of \cite{BNII}.    See \cite{CGK} for a recent paper 
containing a kind of  `Urysohn lemma with positivity' for function algebras.
Indeed many results in Sections 2--6 apply immediately to function algebras (uniform algebras),
that is to uniformly closed subalgebras of $C(K)$, since these
are special cases of operator algebras.  We will not take the time
to point these out, although some of these applications are new. 

 In Section 7, we construct an interesting new approximately unital operator algebra, and use it to solve questions arising in our earlier work, and 
which we now describe some background for.   We recall that a semisimple Banach algebra $A$ is
a modular annihilator algebra iff no element of $A$ has a nonzero
limit point in its spectrum \cite[Theorem 8.6.4]{Pal}).
 If $A$ is also commutative then this is equivalent to
the Gelfand spectrum of $A$ being discrete \cite[p.\ 400]{LN}. 
We write $M_{a,b} : A \to A : x \mapsto axb$,
where $a, b \in A$.   Recall that a Banach algebra
is {\em compact} if the map $M_{a,a}$ is compact
for all $a \in A$.  We say that $A$ is {\em weakly compact} if
$M_{a,a}$ is weakly compact for all $a \in A$.  If $A$ is approximately unital  and commutative then $A$ is
weakly compact iff $A$ is an  ideal  in its bidual  $A^{**}$.  
(We use the same symbol $*$ for the Banach dual
and for the involution or adjoint operator, the reader will have to determine which
is meant from the context.)
 In the noncommutative case  $A$ is
weakly compact  iff $A$ is a
{\em hereditary subalgebra} (or {\em HSA}, defined below) in its bidual.
It is known  \cite{Pal} that every compact semisimple Banach algebra
is a modular annihilator algebra (and conversely every semisimple `annihilator algebra', or 
more generally any Banach algebra
with dense socle, is compact).  Thus it is of interest to know
if there are any connections for operator algebras 
between being a semisimple modular annihilator algebra, and being
weakly compact.  See the discussion after Proposition 5.6 in \cite{ABR}, where some specific questions
along these lines are raised.  We solve these here; indeed
we have solved almost all open questions posed in our previous papers
\cite{BRI, BRII, ABR}.  In particular we know that 1)\ 
 a semisimple approximately unital 
operator algebra which is a modular annihilator algebra
need not be weakly compact, 2)\ an approximately unital commutative weakly compact 
semisimple operator algebra $A$ need not have countable or scattered spectrum
(in fact the spectrum of some of its elements can have nonempty interior),
and 3) \ a radical operator algebra can have a semisimple bidual.  
We present the first and second  of these examples  here.   

We now turn to notation and some background facts mostly needed for Sections 2--6.
 In this paper $H$ will always be a Hilbert space, usually the Hilbert space
on which our operator algebra is acting, or is completely isometrically represented.
We recall that by a theorem due to  Ralf Meyer,
every operator algebra $A$ has a  unitization $A^1$ 
which is unique up
to completely isometric homomorphism (see
 \cite[Section 2.1]{BLM}). Below $1$ always refers to
the identity of $A^1$ if $A$ has no identity.   We write 
oa$(x)$  for the operator algebra generated by $x$ in $A$,
the smallest closed subalgebra containing $x$.
A {\em state} of an approximately unital
operator algebra $A$ is a functional with $\Vert \varphi \Vert = \lim_t \, \varphi(e_t) = 1$ 
for some (or any) cai $(e_t)$ for $A$.  These extend to states of $A^1$.  See \cite[Section 2.1]{BLM}
for details.    With this in mind, we define a state on any nonunital operator algebra $A$  
to be a norm $1$ functional that extends to a state
on $A^1$.
We  write $S(A)$ for the collection of such states; this is the  {\em state space} of $A$.   These 
extend further by the Hahn-Banach theorem to a state
on any $C^*$-algebra generated by $A^1$, and therefore restrict 
to a positive functional on any $C^*$-algebra $B$ generated by $A$.
The latter restriction is actually a state, since it has norm $1$
(even on  $A$).
Conversely, every state on $B$ 
extends to a state on $B^1$, and this restricts to a state on $A^1$.
From these considerations it is easy to see that states on an operator algebra $A$ may equivalently 
be defined to be norm $1$ functionals that extend to a state
on any $C^*$-algebra $B$ generated by $A$.

For us a {\em projection}
is always an orthogonal projection, and an {\em idempotent} merely
satisfies $x^2 = x$. If $X, Y$ are sets, then $XY$ denotes the
closure of the span of products of the form $xy$ for $x \in X, y \in
Y$.   We write $X_+$ for the positive operators (in the usual sense) that happen to
belong to $X$. We write $M_n(X)$ for the space of $n \times n$ matrices 
over $X$, and of course $M_n = M_n(\Cdb)$.  
  The
second dual $A^{**}$ is also an operator algebra with its (unique)
Arens product, this is also the product inherited from the von Neumann
algebra $B^{**}$ if
$A$ is a subalgebra of a $C^*$-algebra $B$. 
Note that
$A$ has a cai iff $A^{**}$ has an identity $1_{A^{**}}$ of norm $1$,
and then $A^1$ is sometimes identified with $A + \Cdb 1_{A^{**}}$.
In this case the multiplier algebra $M(A)$ is identified with the
idealizer of $A$ in $A^{**}$ (that is, the set of elements
$\alpha\in A^{**}$ such that $\alpha A\subset A$ and $A
\alpha\subset A$).
 It can also be viewed as the
idealizer of $A$ in $B(H)$, if the above representation  on $H$ is  nondegenerate.

For an operator algebra, not necessarily approximately  unital,
we recall that $\frac{1}{2} {\mathfrak F}_A = \{ a \in A : \Vert 1 - 2 a \Vert \leq 1 \}$.   Here $1$ is
the identity of the unitization $A^1$ if $A$ is nonunital.  
As we said, $A^1$ is uniquely defined, and can be viewed 
as $A + \Cdb I_H$ if $A$ is completely isometrically represented 
as a subalgebra of $B(H)$.  Hence so is
$A^1 + (A^1)^*$ uniquely defined, by e.g.\ 1.3.7 in \cite{BLM}.  
We define $A + A^*$ to be the obvious subspace
of $A^1 + (A^1)^*$.  This is well defined independently 
of the particular Hilbert space $H$ on which $A$
is represented, as shown at the start of Section 3 in \cite{BRII}. 
 Thus a statement such as
$a + b^* \geq 0$ makes sense whenever $a, b \in A$, and is independent of the particular $H$ on which $A$
is represented.
Hence the set ${\mathfrak r}_A = \{ a \in A : a + a^* \geq 0 \}$ is
independent of the particular representation too.    Elements in ${\mathfrak r}_A$, that is elements in $A$ with
${\rm Re}(x) \geq 0$  will sometimes be called {\em accretive}.

We recall that an {\em r-ideal} is a right ideal with a left cai, and an {\em $\ell$-ideal} is a left ideal with a right cai.
We say that an operator algebra $D$ with cai, which is a subalgebra of
another operator algebra $A$, is a HSA (hereditary subalgebra)
in $A$, if $DAD \subset D$.    See
\cite{BHN} for the basic theory of HSA's.
HSA's in $A$ are in an order preserving,
bijective correspondence with the r-ideals in $A$, and
 with the $\ell$-ideals in $A$.   Because of this symmetry
we will usually restrict our results to the r-ideal case; the $\ell$-ideal case will be analogous.
There is also a bijective correspondence with the
{\em open projections} $p \in A^{**}$, by which we mean that there
is a net $x_t \in A$ with $x_t = p x_t  \to p$ weak*,
or equivalently with $x_t = p x_t p \to p$ weak* (see  \cite[Theorem 2.4]{BHN}).  These are
also the open projections $p$ in the sense of Akemann \cite{Ake2} in $B^{**}$, where $B$ is a $C^*$-algebra containing $A$, such that
$p \in A^{\perp \perp}$.   If $A$ is approximately unital then the complement $p^\perp = 1_{A^{**}} - p$
 of an open projection for $A$
 is called a {\em closed projection} for $A$.   A closed projection $q$ for which there exists an $a \in {\rm Ball}(A)$ with 
$aq = qa = q$ is called {\em compact}.  This is equivalent to 
$A$ being a closed projection with respect to 
$A^1$, if $A$ is approximately unital.  See \cite{BNII, BRII} for the theory of compact projections in operator algebras.  

If $x \in {\mathfrak r}_A$ then it is shown in \cite[Section 3]{BRII}  that the operator algebra oa$(x)$ generated by $x$ in $A$ has a cai, which can be taken to be a normalization of $(x^{\frac{1}{n}})$, and the weak* limit of $(x^{\frac{1}{n}})$  is the support projection $s(x)$ for $x$.  This is an open projection.
 We recall that if  $x \in \frac{1}{2} \, {\mathfrak F}_A$ then the {\em peak projection} associated with 
$x$  is   $u(x) = \lim_n \, x^n$ (weak* limit).  We have $u(x^{\frac{1}{n}}) = u(x)$, for 
$x \in \frac{1}{2} \, {\mathfrak F}_A$ (see \cite[Corollary 3.3]{BNII}).  Compact projections
in approximately unital algebras are precisely the infima (or decreasing weak* limits) of collections of such
peak projections  \cite{BNII}.  Note that $x \in {\mathfrak c}_A = \Rdb_+ {\mathfrak F}_A$  iff there is a positive
constant $C$ with $x^* x \leq C(x+x^*)$.    

In this paper we will sometimes use the word `cigar' for 
the wedge-shaped region consisting of numbers $re^{i \theta}$ with
argument $\theta$ such that $|\theta| < \rho$ (for some fixed small $\rho > 0$), 
which are also inside
the circle $|z - \frac{1}{2}| \leq \frac{1}{2}$.
If $\rho$ is small enough so that $|{\rm Im}(z)| < \epsilon/2$ for all $z$ in this
region,  then we will call this a `horizontal cigar of height $< \epsilon$ centered on the 
line segment $[0,1]$ in the $x$-axis'.

 By {\em numerical range}, we will mean the one defined by states, while the literature we quote
usually uses the one defined by vector states on $B(H)$.  However since the former range is the 
closure of the latter, as is well known, this will cause no difficulties.  
For any operator $T \in B(H)$ whose numerical range does not include strictly negative
numbers, and for any $\alpha \in [0,1]$, there is a well-defined `principal' root $T^\alpha$,
which obeys the usual law $T^\alpha T^\beta = T^{\alpha + \beta}$ if $\alpha + \beta \leq 1$
(see e.g.\ \cite{MP,LRS}).  If the numerical range   
is contained in a sector $S_\psi = \{ r e^{i \theta} : 0 \leq r , \, \text{and} \,  
-\psi \leq \theta \leq \psi \}$ where $0 \leq \psi < \pi$, then things are better still.  For fixed $\alpha \in (0,1]$
there is a constant $K > 0$ with $\Vert T^\alpha - S^\alpha \Vert \leq K \Vert  T - S \Vert^\alpha$ for operators $S, T$ with numerical range
in $S_\psi$ (see \cite{MP,LRS}).   Our operators $T$ will in fact be accretive (that is,
 $\psi \leq \frac{\pi}{2}$), 
and then these                 
powers obey the usual laws such as  $T^\alpha T^\beta = T^{\alpha + \beta}$
for all $\alpha, \beta > 0$,
$(T^\alpha)^\beta
= T^{\alpha \beta}$ for $\alpha \in (0,1]$ and any $\beta > 0$, 
and $(T^*)^\alpha  = (T^\alpha)^*$.     We shall see in Lemma \ref{rootf} that if  $\psi < \frac{\pi}{2}$ then $T
\in {\mathfrak c}_{B(H)}$.  The numerical range 
of $T^\alpha$ lies in $S_{\alpha \frac{\pi}{2}}$  for any $\alpha \in (0,1)$.  Indeed if  $n \in \Ndb$ then $T^{\frac{1}{n}}$ is the unique $n$th root of $T$ with numerical 
range in $S_{\frac{\pi}{2 n}}$.   See e.g.\ \cite[Chapter IV, Section 5]{NF} and \cite{Haase}
for all of these facts.   Some of the following facts are no doubt also in the literature, since we do not know of a reference
we sketch short proofs.

\begin{lemma} \label{roots}   For an accretive operator $T \in B(H)$ we have:
\begin{itemize} \item [(1)]   $(cT)^\alpha  = c^\alpha T^\alpha$ for positive scalars $c$, and $\alpha \geq 0$.
\item [(2)]  $\alpha \mapsto T^{\alpha}$ 
is continuous on $(0,\infty)$.
\item [(3)]    $T^\alpha \in {\rm oa}(T)$,  the operator algebra generated by $T$, if $\alpha > 0$.
\end{itemize} 
\end{lemma}  \begin{proof}   (1) \ This is obvious if $\alpha = \frac{1}{n}$ for $n \in \Ndb$ by  the uniqueness of $n$th roots discussed above. 
In general it can be proved e.g.\ by a change of variable in the Balakrishnan representation for powers
(see e.g.\ \cite{Haase}).

(2) \ By a triangle inequality argument, and the inequality for 
$\Vert T^\alpha - S^\alpha \Vert$ above, we may assume that $T \in {\mathfrak c}_{B(H)}$.    By (1) we may assume that
$T \in  \frac{1}{2} {\mathfrak F}_{B(H)}$.    Define $$f(z) = ((1-z)/2)^{\alpha} - ((1-z)/2)^{\beta} \; ,  \qquad z \in \Cdb, |z| \leq 1.$$
Via the relation $T^\alpha T^\beta = T^{\alpha + \beta}$ above, we may assume that $\beta \in (0,1]$.
Fix such $\beta$.  By complex numbers one can show that  $|f(z)| \leq g(|\alpha - \beta|)$  on the unit disk,
for a function $g$ with $\lim_{t \to 0^+} \, g(t) = 0$.   By von Neumann's inequality,
  used as in \cite[Proposition 2.3]{BRII}, we have $$\Vert T^\alpha - T^\beta \Vert= \Vert f(1 - 2T) \Vert  \leq g(|\alpha - \beta|).$$ 
Now let $\alpha \to \beta$.  

(3) \ We proved this  in the second paragraph 
of \cite[Section 3]{BRII}  if $\alpha = \frac{1}{n}$ for $n \in \Ndb$.   Hence for $m \in \Ndb$ we have by the paragraph above
the lemma that $T^{\frac{m}{n}} = (T^{\frac{1}{n}})^m \in {\rm oa}(T)$.  The general case for $\alpha > 0$ then follows by the continuity in (2).
\end{proof}

In particular, ${\mathfrak r}_A$ is closed under taking roots for any operator algebra $A$.  

\section{Positivity in operator algebras}

In earlier papers \cite{BRI,BRII,Read} we introduced and studied a new notion of positivity in operator algebras.
In this section and the next several sections,
 we study foundational aspects of this positivity, and of the associated ordering, which we call the 
${\mathfrak r}$-{\em ordering}.

Let $A$ be an operator algebra, not necessarily approximately  unital for the present.
Note that ${\mathfrak r}_A = \{ a \in A : a + a^* \geq 0 \}$
is a closed cone in $A$, hence is Archimedean, but it is not 
proper (hence is what is sometimes called a {\em wedge}).  On the other hand
${\mathfrak c}_A = \Rdb_+ {\mathfrak F}_A$ is not closed in general, but it is a a proper cone
(that is, ${\mathfrak c}_A \cap (-{\mathfrak c}_A) = (0)$).
Indeed suppose $a \in {\mathfrak c}_A \cap (-{\mathfrak c}_A)$.  Then  $\Vert 1 - t a \Vert
\leq 1$ and $\Vert 1 + s a \Vert
\leq 1$ for some $s, t > 0$.  By convexity we may assume $s = t$ (by replacing them by
$\min \{ s , t \}$).   It is well known that in any
Banach algebra with an identity of norm $1$, the identity  is an extreme point of the ball.
Applying this in $A^1$ we deduce that  $a = 0$ as desired.

The ${\mathfrak r}$-{\em ordering} is simply the order induced by the above closed cone; that is
$b$ is `dominated' by $a$ iff $a - b \in {\mathfrak r}_A$.  
If $A$ is a subalgebra of an operator algebra $B$, it is clear from
a fact mentioned in the introduction (or at the start of \cite[Section 3]{BRII}) that the positivity of $a + a^*$ may be computed
with reference to any containing $C^*$-algebra, that ${\mathfrak r}_A \subset {\mathfrak r}_B$.
If $A, B$ are
approximately unital subalgebras of $B(H)$ then it follows from \cite[Corollary 4.3 (2)]{BRII}
that $A \subset B$ iff ${\mathfrak r}_A \subset {\mathfrak r}_B$.   As in \cite[Section 8]{BRI},
${\mathfrak r}_A$ contains no idempotents which are not orthogonal projections,
and no nonunitary isometries $u$ (since by the analogue 
of \cite[Corollary 2.8]{BRI} we would have $u u^* = s(u u^*) = s(u^* u) = I$).  
In \cite{BRII} it is shown that $\overline{{\mathfrak c}_A} = {\mathfrak r}_A$.

\begin{lemma}  For
any operator algebra $A$,  $x \in {\mathfrak r}_A$  iff ${\rm Re}(\varphi(x)) \geq 0$ for all
states $\varphi$ of $A^1$.
\end{lemma}  \begin{proof}
 Such $\varphi$
extend to states on $C^*(A^1)$.
 So we may assume that $A$ is a
unital $C^*$-algebra, in which case the result is well known ($x + x^* \geq 0$ iff
$2 {\rm Re}(\varphi(x)) = \varphi(x+x^*) \geq 0$ for all
states $\varphi$).
\end{proof}

{\bf Remark.}  For an operator algebra which is not approximately unital, it is not true that $x \in {\mathfrak r}_A$  iff ${\rm Re}(\varphi(x)) \geq 0$ for all
states $\varphi$ of $A$, with states defined as in the introduction.  An example would be $\Cdb \oplus \Cdb$, with the second summand given the 
zero multiplication. 

\medskip

If $A$ is unital, then $A= {\mathfrak r}_A - {\mathfrak r}_A$; indeed any 
$a \in {\rm Ball}(A)$ may be written as $\frac{1}{2}(1 + a) - \frac{1}{2}(1 -a) \in
\frac{1}{2} {\mathfrak F}_A - \frac{1}{2} {\mathfrak F}_A$.    Hence
if $A$ is approximately unital, then $A^{**} = {\mathfrak r}_{A^{**}} - {\mathfrak r}_{A^{**}}$.
Since the weak* closure of ${\mathfrak r}_A$ is ${\mathfrak r}_{A^{**}}$ (see \cite[Corollary 3.6]{BRII}),
it follows by some variant of 
the Hahn-Banach theorem that the norm closure of ${\mathfrak r}_A - {\mathfrak r}_A$
is $A$.  We shall see next that the norm closure is unnecessary.

\begin{theorem} \label{dualcor}  Let $A$ be an approximately unital operator algebra.
Any 
$x \in A$ with $\Vert x \Vert < 1$ may be written as $x = a-b$ with $a, b \in {\mathfrak r}_A$
and $\Vert a \Vert < 1$ and $\Vert b \Vert < 1$.
In fact 
one may choose such $a, b$ to also be in $\frac{1}{2} {\mathfrak F}_{A}$.   
\end{theorem}

\begin{proof}   Assume that $\Vert x \Vert = 1$.
By what we said above  the theorem, 
$x = \eta - \xi$ for $\eta, \xi \in \frac{1}{2} {\mathfrak F}_{A^{**}}$.
By \cite[Lemma 8.1]{BRI}
we deduce that $x$ is in the weak closure of the convex set
$\frac{1}{2} {\mathfrak F}_{A} - \frac{1}{2} {\mathfrak F}_{A}$.
Therefore it is in the norm closure, so given $\epsilon > 0$ there exists 
$a_0, b_0 \in \frac{1}{2} {\mathfrak F}_{A}$ with 
$\Vert x - (a_0 - b_0) \Vert < \frac{\epsilon}{2}$. 
Similarly, there exists
$a_1, b_1 \in \frac{1}{2} {\mathfrak F}_{A}$ with
 $\Vert x - (a_0 - b_0) - \frac{\epsilon}{2} (a_1 - b_1) \Vert < \frac{\epsilon}{2^2}$.
Continuing in this manner, one produces sequences $(a_k), (b_k)$ in $\frac{1}{2} {\mathfrak F}_{A}$. 
Setting $a' = \sum_{k=1}^\infty \, \frac{1}{2^k} \, a_k$ and 
$b' = \sum_{k=1}^\infty \, \frac{1}{2^k} \, b_k$, which are in $\frac{1}{2} {\mathfrak F}_{A}$ since the 
latter is a closed convex set,
we have $x = (a_0 - b_0) + \epsilon (a' - b')$.  Let $a = a_0 + \epsilon a'$ and 
$b = b_0 + \epsilon b'$.    By convexity $\frac{1}{1 + \epsilon} a \in \frac{1}{2} {\mathfrak F}_{A}$
and $\frac{1}{1 + \epsilon} b \in \frac{1}{2} {\mathfrak F}_{A}$.

If $\Vert x \Vert < 1$ choose  $\epsilon > 0$  with $\Vert x \Vert (1 + \epsilon) < 1$.
Then $x/\Vert x \Vert = a - b$ as above, so that 
$x = \Vert x \Vert \, a - \Vert x \Vert \, b$.    We have $$\Vert x \Vert \, a
= (\Vert x \Vert (1 + \epsilon)) \cdot (\frac{1}{1 + \epsilon} a ) \in [0,1) \cdot \frac{1}{2} {\mathfrak F}_{A}
\subset \frac{1}{2} {\mathfrak F}_{A} ,$$
and similarly $\Vert x \Vert \, b \in  \frac{1}{2} {\mathfrak F}_{A}$.  
  \end{proof}

In the language of ordered Banach spaces, the above shows that ${\mathfrak r}_A$ 
and ${\mathfrak c}_A$ are  {\em generating} cones (this is sometimes 
called {\em positively generating} or {\em directed} or {\em co-normal}). 

\bigskip

{\bf Remarks.}   1) \ Can every  
$x \in {\rm Ball}(A)$   be written as $x = a-b$ with $a,b
 \in {\mathfrak r}_A \cap {\rm Ball}(A)$?
  As we said above, this is true if $A$ is unital.
We can show that in general $x \in {\rm Ball}(A)$ cannot be written  as $x = a-b$ with  
$a, b \in \frac{1}{2} {\mathfrak F}_{A}$.   To see this let $A$ be the set 
of functions in the  disk algebra vanishing at $-1$, an approximately unital function 
algebra.   Let $W$ be the  closed connected set obtained from the unit disk by removing the 
`slice' consisting of all complex numbers with negative real part and argument
in a small open interval containing $\pi$.  By the Riemann mapping theorem
 it is easy to see that there is a conformal map $h$ of the disk onto
$W$ taking $-1$ to $0$, so that $h \in {\rm Ball}(A)$.   By way of contradiction
suppose that  
$h = a-b$ with $a,b \in \frac{1}{2} {\mathfrak F}_{A}$.  Then it is easy to see
from the geometry of the circles $B(0,1)$ and $B(\frac{1}{2}, \frac{1}{2})$, that $a + b = 1$
on a nontrivial arc of the unit circle, and hence everywhere.  
However $a(-1) + b(-1) = 0$, which is the desired contradiction.

\medskip

2) \ Applying Theorem \ref{dualcor} to $ix$ for $x  \in A$, one gets a similar decomposition
$x = a-b$ with the `imaginary parts' of $a$ and $b$ positive.   One might ask if, as is suggested by 
the $C^*$-algebra case, one may write for each $\epsilon$, any $x \in A$ with $\Vert x \Vert < 1$
as $a_1 - a_2 + i(a_3 - a_4)$ for $a_k$ with numerical range in a thin horizontal cigar of height $<
\epsilon$ centered on the line segment $[0,1]$ in the $x$-axis.  In fact this is false, as one can see
in the case that $A$ is  the set of upper triangular $2 \times 2$ 
matrices with constant diagonal entries.

\begin{corollary} \label{Ahasc}   An operator algebra $A$ has a cai iff 
$A = {\mathfrak c}_A - {\mathfrak c}_A$.  \end{corollary}

\begin{proof}  This follows from Theorem \ref{dualcor} and 
\cite[Corollary 4.3]{BRII}.  \end{proof}

Most of the results in this section apply to approximately unital operator algebras.  We offer a couple of results 
that are useful in applying the approximately unital  case to algebras with no approximate identity.  We 
will use the space $A_H$ studied in \cite[Section 4]{BRII}, this is the largest approximately unital subalgebra of $A$;
it is actually a  HSA in $A$ (and will be an ideal if $A$ is commutative).

\begin{corollary} \label{Ahasc2}  For any operator algebra $A$,
$$A_H = {\mathfrak r}_A - {\mathfrak r}_A = {\mathfrak c}_A - {\mathfrak c}_A.$$
In particular these spaces are  closed, and form a HSA of $A$.
 \end{corollary}

\begin{proof} 
In the language of \cite[Section 4]{BRII}, and using
\cite[Corollary 4.3]{BRII}, ${\mathfrak r}_A = {\mathfrak r}_{A_H}$, 
and $$A_H = \overline{{\rm Span}}({\mathfrak r}_{A_H}) = {\mathfrak r}_{A_H} - {\mathfrak r}_{A_H}
= {\mathfrak r}_A - {\mathfrak r}_A,$$ by Theorem \ref{dualcor}.  
A similar argument works with ${\mathfrak r}_{A_H}$ replaced by
${\mathfrak c}_{A_H}$ using Corollary \ref{Ahasc} and facts from
\cite[Section 4]{BRII} about ${\mathfrak F}_{A_H}$.
\end{proof}

\begin{lemma} \label{mnah}  Let $A$ be any operator algebra.  Then   for every $n \in \Ndb$, 
$$M_n(A_H) = M_n(A)_H \; , \; \; \; \; \;  {\mathfrak r}_{M_n(A)} = {\mathfrak r}_{M_n(A_H)}
 \; , \; \; \; \; \;  {\mathfrak F}_{M_n(A)} = {\mathfrak F}_{M_n(A_H)}$$
(these are the matrix spaces).  
\end{lemma}  

\begin{proof}   Clearly $M_n(A_H)$ is an approximately unital subalgebra of $M_n(A)$.  So $M_n(A_H)$ is contained in $M_n(A)_H$,
 since the latter is the largest approximately unital subalgebra of $M_n(A)$.    
To show that $M_n(A)_H \subset M_n(A_H)$ it suffices by 
Corollary \ref{Ahasc2} to show that ${\mathfrak r}_{M_n(A)} \subset M_n(A_H)$.  So  suppose that $a = [a_{ij}] \in M_n(A)$
with $a + a^* \geq 0$.  Then $a_{ii} + a_{ii}^* \geq 0$ for each $i$.  We also have 
$\sum_{i,j} \, \bar{z_i} \, (a_{ij} + a_{ji}^*) \, z_j \geq 0$ for all scalars 
$z_1, \cdots , z_n$.   So $\sum_{i,j} \, \bar{z_i} \, a_{ij} z_j \in {\mathfrak r}_A$.
 Fix an $i, j$, which we will assume to be $1, 2$ for simplicity.
Set all $z_k = 0$ if $k \notin \{ i, j \} = \{ 1, 2 \}$, to deduce 
$$\bar{z_1} z_2 a_{12}  + \bar{z_2} z_1 a_{21} = \sum_{i,j = 1}^2 \, \bar{z_i} \, a_{ij} z_j \, - \, (|z_1|^2 a_{11} +
|z_2|^2 a_{22}) \in {\mathfrak r}_A - {\mathfrak r}_A = A_H.$$
Choose $z_1 = 1$; if  
$z_2 = 1$ then $a_{12}  + a_{21} \in A_H$, while if $z_2 = i$ then 
$i(a_{12}  -  a_{21}) \in A_H$.    So $a_{12}, a_{21} \in A_H$.  A similar argument 
shows  that $a_{ij}  \in A_H$ for all $i, j$.  Thus $M_n(A_H) = M_n(A)_H$, from which we deduce
by \cite[Corollary 4.3 (1)]{BRII} that
$${\mathfrak r}_{M_n(A)} = {\mathfrak r}_{M_n(A)_H} = {\mathfrak r}_{M_n(A_H)}.$$
Similarly ${\mathfrak F}_{M_n(A)} = {\mathfrak F}_{M_n(A)_H}
= {\mathfrak F}_{M_n(A_H)}$.
\end{proof}

The last result is used in \cite{BBS}.
 
We write ${\mathfrak c}^{\Rdb}_{A^*}$ for the real dual cone of ${\mathfrak r}_A$, the 
set of continuous $\Rdb$-linear $\varphi : A \to \Rdb$ such that $\varphi({\mathfrak r}_A) \subset [0,\infty)$.
Since $\overline{{\mathfrak c}_A} = {\mathfrak r}_A$ this is also the real dual cone of ${\mathfrak c}_A$.

We collect some facts about real states.   Much of this parallels the development 
of ordinary states, but we include a brief sketch of
the details to  save others having to check these
 each time they are needed
in the future.  
In the rest of this section $A$ is an approximately unital operator algebra.
A {\em real state} on $A$ will be a contractive  $\Rdb$-linear $\Rdb$-valued functional on $A$
such that $\varphi(e_t) \to 1$ for some cai $(e_t)$ of $A$.  This is equivalent to
$\varphi^{**}(1) = 1$, where $\varphi^{**}$ is the canonical 
$\Rdb$-linear extension to $A^{**}$, and $1$ is the identity of 
$A^{**}$ (here we are using the canonical identification between real second duals
and complex second duals of a complex Banach space \cite{Li}).  Hence $\varphi(e_t) \to 1$ for every
cai $(e_t)$ of $A$.  

Since we can identify $A^1$ with $A + \Cdb 1_{A^{**}}$ if we like, by the 
last paragraph it follows that  real states of $A$ extend to real states of $A^1$, hence by the 
Hahn-Banach theorem they extend to real states of $C^*(A^1)$.
We claim that a  real state $\psi$ on a $C^*$-algebra $B$ is positive on $B_+$, and
is zero  on $i B_+$.
To see this, we may assume that $B$ is a von Neumann algebra
(by extending the state to its second dual similarly to as in the last paragraph).  For any projection $p \in B$,
$C^*(1,p) \cong \ell^\infty_2$, and it is an easy exercise to see that 
real states on $\ell^\infty_2$ are positive on $(\ell^\infty_2)_+$
and are zero on $i (\ell^\infty_2)_+$.
Thus $\psi(p) \geq 0$ and $\psi(ip) = 0$ for any projection $p$, hence 
$\psi$ is positive on $B_+$  and zero  on $i B_+$ by the Krein-Milman theorem.   

We deduce:

\begin{lemma}  \label{duals}  Real states on an approximately unital operator algebra
$A$ are in ${\mathfrak c}^{\Rdb}_{A^*}$. 
\end{lemma}

\begin{proof}
If $a + a^* \geq 0$, and $\tilde{\varphi}$ is the 
real state extension above to $B = C^*(A^1)$, 
 then $$\varphi(a) = \frac{1}{2} \tilde{\varphi}(a + a^*) + 
\frac{1}{2} \tilde{\varphi}(-i \cdot i (a - a^*)) = 
\frac{1}{2} \tilde{\varphi}(a + a^*) \geq 0 ,$$ since 
$i(a-a^*) \in B_{\rm sa} = B_+ - B_+$, and $\tilde{\varphi}(i(B_+ - B_+)) = 0$, as we said
above.
 \end{proof}

\begin{lemma}  \label{dualc}    Suppose that $A$ is an approximately unital operator algebra.
 The real dual cone ${\mathfrak c}^{\Rdb}_{A^*}$  equals 
$\{ t \, {\rm Re}(\psi) : \psi \in S(A) , \, t \in [0,\infty) \}$. 
It also equals the set of restrictions to $A$ of the real parts of
positive functionals on any $C^*$-algebra containing 
(a copy of) $A$ as a closed subalgebra.    Also, ${\mathfrak c}^{\Rdb}_{A^*}$ is a   proper cone. \end{lemma}  

\begin{proof}  
To see that ${\mathfrak c}^{\Rdb}_{A^*}$  is a proper cone,
let  $\rho, -\rho \in {\mathfrak c}^{\Rdb}_{A^*}$.  Then
$\rho(a) = 0$ for all $a \in {\mathfrak r}_A$, and hence $\rho = 0$
by the fact above that the norm closure of ${\mathfrak r}_A - {\mathfrak r}_A$
is $A$.

The restriction to $A$ of the real part of
any positive functional on a $C^*$-algebra containing $A$, is easily seen
to be in ${\mathfrak c}^{\Rdb}_{A^*}$.  For the converse, if $\varphi \in {\mathfrak c}^{\Rdb}_{A^*}$ define 
$\tilde{\varphi}(a) = \varphi(a) - i \varphi(ia)$ for $a \in A$.
As is checked in a basic functional analysis course, $\tilde{\varphi} 
\in A^*$ (the complex dual space), and note that 
$\tilde{\varphi}({\mathfrak r}_A) \subset {\mathfrak r}_{\Cdb}$.  This is also true 
at the matrix level, since in the notation of the proof of 
Lemma \ref{mnah}, $$\sum_{i,j} \overline{z_i} \varphi(a_{ij} + a_{ji}^*) z_j 
= \varphi(\sum_{i,j} \overline{z_i} (a_{ij} + a_{ji}^*) z_j) \geq 0 .$$
That is, $\varphi$ is real completely positive in the sense of \cite[Section 2]{BBS},
and so by that paper
$\tilde{\varphi}$ extends to a positive functional $\psi$ on $C^*(A)$,
and $\varphi = {\rm Re}(\psi)$.  

The remaining statements follow
from the well known facts that, first, 
states of an approximately unital 
operator algebra $A$ extend to states of any $C^*$-algebra generated by $A$,
as we said in the introduction,
and second, positive functionals on any $C^*$-algebra $B$  
are just the nonnegative multiples of states on $B$.   \end{proof}  

Below we will use silently a couple of times the obvious fact that 
if $\varphi, \psi$ are two complex valued functionals on $A$ with
${\rm Re}(\varphi(a)) = {\rm Re}(\psi(a))$ for all $a \in A$,  
then 
$\varphi = \psi$ on $A$. 

\begin{lemma}  \label{dualc2}  Suppose that $A$ is an approximately unital operator algebra.
\begin{itemize} \item [(1)]   If $f \leq g \leq h$   in $B(A,\Rdb)$ in 
the ${\mathfrak c}_{A^*}$-ordering,
then $\Vert g \Vert \leq 2 \max \{ \Vert f \Vert , \Vert h \Vert \}$.
\item [(2)]   The cone  ${\mathfrak c}_{A^*}$ is 
 {\em additive} (that is, the norm on $B(A,\Rdb)$ is additive on ${\mathfrak c}^{\Rdb}_{A^*}$).
\item [(3)]  If $(\varphi_t)$ is an increasing net in 
${\mathfrak c}^{\Rdb}_{A^*}$ which is bounded in norm,
then the net converges in norm, and its limit is the least upper 
bound of the net.  \end{itemize} 
\end{lemma}  \begin{proof}   (1) \ This may be deduced e.g.\
from  Theorem \ref{dualcor} by \cite[Theorem 1.5]{AE}.

(2) \ If $\psi$ is a positive map on $C^*(A)$ and $(e_t)$ is a cai for $A$ then then 
$$\Vert \psi \Vert = \lim_t \psi(e_t)  = \lim_t \, {\rm Re} \, \psi(e_t)
\leq \Vert {\rm Re} \, \psi_{\vert A} \Vert \leq \Vert \psi_{\vert A} \Vert 
\leq \Vert \psi \Vert.$$
Hence $\Vert \varphi \Vert = \langle 1 , \varphi \rangle$ for all
$\varphi  \in {\mathfrak c}^{\Rdb}_{A^*}$, and it follows
that the norm on $B(A,\Rdb)$ is additive on ${\mathfrak c}^{\Rdb}_{A^*}$. 

(3) \ Follows from \cite[Proposition 3.2]{AE}.
\end{proof}  

{\bf Remarks.}  1) \   In the language of ordered Banach spaces, the norm estimate in Lemma \ref{dualc2} (1)
  is saying that 
 ${\mathfrak c}_{A^*}$ is a (2-){\em normal cone}.  The best constant here is $2$; it 
is not hard to show there exist 
approximately unital algebras $A$ with ${\mathfrak c}^{\Rdb}_{A^*}$ 
not $C$-normal for any $C < 2$.
Indeed if it was, then ${\mathfrak r}_{A}$ would be $C'$-directed 
for some $C' < 2$ by \cite[Theorem 1.5]{AE}.
However if $A = A(\Ddb)$, the  disk algebra, it is easy to see that if $z = f-g$
for $f, g \in A$ having positive real part, then Re$(f(1)) \geq 1$.
Hence $\Vert f \Vert \geq 1$.  Similarly, $\Vert g \Vert \geq 1$, 
so that $\Vert f \Vert + \Vert g \Vert \geq 2$.

\smallskip

2) \ It is probably never  true for 
an approximately unital operator algebra $A$ 
that 
$B(A,\Rdb) = {\mathfrak c}^{\Rdb}_{A^*} - {\mathfrak c}^{\Rdb}_{A^*}$.
Indeed, in the case $A = \Cdb$ the latter space has  real dimension $1$.
However the complex span of the (usual) states of an
 approximately unital operator algebra $A$ is $A^*$ (the complex dual space).  
Indeed by a result of Moore \cite{Moore,AE2}, the 
complex span of the states of any unital Banach algebra $A$ is $A^*$.
In the approximately unital operator algebra case one can prove this same fact by 
a reduction to the unital case by noting that by the argument in the second last 
paragraph of \cite[Section 4]{BRII}, the span of the states on $A$ is
the span of the restrictions to $A$ of the states on $A^1$.

\smallskip

3)   Every element $x \in \frac{1}{2} 
{\mathfrak F}_A$ need not achieve its norm at a state, even in $M_2$ (consider 
$x = (I + E_{12})/2$ for example).

\begin{corollary}  \label{rsaj}  The real states on an approximately unital operator algebra
$A$ are just the real parts of ordinary states on $A$. \end{corollary}   \begin{proof} Certainly 
the real part of an ordinary state is a real state.  If 
$\varphi$ is a real state on $A$, then by Lemma \ref{duals} and  Lemma \ref{dualc}
we have $\varphi = t {\rm Re} \, \psi$ for a state $\psi$ on $A$ which is the 
restriction of a state on $C^*(A)$.   
In the proof of Lemma \ref{dualc2} we saw that $\Vert {\rm Re} \, \psi \Vert = 1$,
so that $t = 1$.  
  \end{proof}  

\begin{lemma}  \label{uext}  Any real state on an approximately unital closed 
subalgebra $A$ of an approximately unital  operator algebra $B$ extends 
to a real state on $B$.  If $A$ is a HSA in $B$ then this extension is unique.
\end{lemma}

\begin{proof}
The first part f
is as in \cite[Proposition 3.1.6]{Ped}.
Suppose that  $A$ is a HSA in $B$ and that $\varphi_1, \varphi_2$ 
are real states  on $B$ extending a real state on $A$.  By the above we may
write $\varphi_i = {\rm Re} \, \psi_i$ for ordinary states  on $B$.
Since $\varphi_1 = \varphi_2$ on $A$ we have 
$\psi_1 = \psi_2$ on $A$.  Hence $\psi_1 = \psi_2$ on $B$ by 
\cite[Theorem 2.10]{BHN}.  So $\varphi_1 = \varphi_2$ on $B$.
  \end{proof}

\begin{corollary} \label{dualco1}  
Let $A$ be an approximately unital operator algebra.
The second dual cone of ${\mathfrak r}_A$ 
(that is, the (real)  dual cone of ${\mathfrak c}^{\Rdb}_{A^*}$) is ${\mathfrak r}_{A^{**}}$.
The (real) predual cone of the dual cone ${\mathfrak c}^{\Rdb}_{A^*}$
is ${\mathfrak r}_A$.
\end{corollary}

\begin{proof}  We use Lemma \ref{dualc}.  Suppose that $a \in A$ with ${\rm Re} \, \varphi(a) \geq 0$
for all states $\varphi$ on $C^*(A)$.   Then $\varphi(a + a^*) \geq 0$ 
for all states $\varphi$, so that $a + a^* \geq 0$.   

A similar proof works if $a \in A^{**}$ to yield the first assertion.
Indeed if Re$ \, a(\varphi) \geq 0$ for all states $\varphi$ on $A$,  
or equivalently if $(a + a^*)(\varphi) \geq 0$ for all states $\varphi$ on $C^*(A)$, then 
$a + a^* \geq 0$ in $C^*(A)^{**}$, so that $a \in {\mathfrak r}_{A^{**}}$. 
\end{proof}

 It follows immediately from Theorem \ref{dualcor} that any approximately 
unital operator algebra  
$A$ is a directed set.  That is, if $x, y \in A$ then there
exists $z \in A$ with $z-x, z-y \in {\mathfrak r}_A$.  In fact more is true.
We recall that the positive part of the  open unit ball of a $C^*$-algebra
is a directed set.  The following generalizes this to operator algebras:
  
\begin{corollary} \label{dirset} If $A$ is an approximately
unital operator algebra then the open unit ball of $A$ 
is a directed set with respect to the ${\mathfrak r}$-ordering.  That is, 
if $x, y \in A$ with $\Vert x \Vert , \Vert y  \Vert < 1$, 
then there 
exists $z \in A$ with $\Vert z\Vert < 1$ and  $z-x, z-y \in {\mathfrak r}_A$.  
 \end{corollary}

\begin{proof}  This follows from Lemma \ref{dualc2} and \cite[Corollary 3.6]{AE}.
\end{proof}
 
For a $C^*$-algebra $B$, a natural ordering on the  open unit ball of $B$ 
turns the latter into a net which is a positive cai for $B$ (see e.g.\ \cite{Ped}).
   We do not know if 
something similar is true for general operator algebras via  Corollary \ref{dirset}.

\section{Some results on positivity in operator algebras} \label{morp} 

In this section we collect several useful facts concerning our  positivity, some of which are used in this 
paper, and some in forthcoming work.  

\subsection{The ${\mathfrak F}$-transform}

In \cite{BRII} the sets  $\frac{1}{2}{\mathfrak F}_{A}$ and ${\mathfrak r}_{A}$ were related by a certain 
transform.  We now establish a few more basic properties of this transform. 
The Cayley transform $\kappa(x) = (x-I)(x+I)^{-1}$ 
of an accretive $x \in A$ exists since $-1 \notin {\rm Sp}(x)$, and
is well known to be a contraction.  Indeed it is well known (see e.g.\ \cite{NF}) 
that if $A$ is unital then the Cayley transform
maps ${\mathfrak r}_A$ bijectively onto the set of contractions in $A$ whose 
spectrum does not contain $1$, and the inverse transform is
$T \mapsto (I+T) (I - T)^{-1}$.  The Cayley transform maps the accretive elements $x$
with ${\rm Re}(x) \geq \epsilon 1$ for some $\epsilon > 0$, onto the set 
of elements $T \in A$ with $\Vert T \Vert < 1$ (see e.g.\ 2.1.14 in \cite{BLM}). 
 The ${\mathfrak F}$-transform ${\mathfrak F}(x) = 1 - (x+1)^{-1} = x (x+1)^{-1}$ may be written
as ${\mathfrak F}(x) = \frac{1}{2} (1 + \kappa(x))$.  Equivalently,
 $\kappa(x) = -(1 - 2 {\mathfrak F}(x))$.   

\begin{lemma}   For any operator algebra $A$, the  ${\mathfrak F}$-transform 
 maps ${\mathfrak r}_{A}$ bijectively onto the set of elements of $\frac{1}{2}{\mathfrak F}_{A}$
of norm $< 1$.   \end{lemma}  \begin{proof}  
First assume that  $A$ is  unital. By the last equations  ${\mathfrak F}({\mathfrak r}_A)$ is 
contained in  the set of elements of $\frac{1}{2}{\mathfrak F}_{A}$
whose spectrum does not contain $1$.  
The inverse of the ${\mathfrak F}$-transform
on this domain is 
$T (I-T)^{-1}$.   To see for example that $T (I-T)^{-1} \in {\mathfrak r}_A$ if $T \in \frac{1}{2}{\mathfrak F}_{A}$
note that 2Re$(T (I-T)^{-1})$ equals
$$(I-T^*)^{-1}(T^*(I-T) + (I-T^*) T) (I-T)^{-1} = (I-T^*)^{-1}(T + T^* - 2T^* T) (I-T)^{-1}$$
which is positive since $T^* T$  is dominated by Re$(T)$ if $T \in \frac{1}{2}{\mathfrak F}_{A}$.  
 Hence for any (possibly nonunital) operator algebra $A$ the  ${\mathfrak F}$-transform 
 maps ${\mathfrak r}_{A^1}$ bijectively onto the set of elements of $\frac{1}{2}{\mathfrak F}_{A^1}$
whose spectrum does not contain $1$.   However this equals the set of elements of $\frac{1}{2}{\mathfrak F}_{A^1}$
of norm $< 1$.     Indeed if $\Vert  {\mathfrak F}(x) \Vert = 1$ then $\Vert  \frac{1}{2} (1 + \kappa(x)) \Vert = 1$, and so $1 - \kappa(x)$ 
is not invertible by \cite[Proposition 3.7]{ABS}.   Hence $1 \in  {\rm Sp}_{A^1}(\kappa(x))$ and $1 \in  {\rm Sp}_A({\mathfrak F}(x))$.
Since ${\mathfrak F}(x) \in A$ iff $x \in A$, we are done.   \end{proof}

Thus in some sense we can identify ${\mathfrak r}_{A}$  with the strict contractions in $\frac{1}{2}{\mathfrak F}_{A}$.  This for example induces an order on this set of strict contractions.

\subsection{Roots of accretive elements}

\begin{lemma} \label{rootf}   Let $A$ be an operator algebra, and $x \in A$. 
\begin{itemize} \item [(1)]   If the numerical range of $x$ is contained in a sector
$S_{\rho}$ for $\rho < \frac{\pi}{2}$
 (see notation above
Lemma {\rm \ref{roots}}), then  $x/\Vert {\rm Re}(x) \Vert
\in \frac{\sec^2 \rho}{2} \, {\mathfrak F}_A$. So $x  \in {\mathfrak c}_A$.
\item [(2)]   
If $x \in {\mathfrak r}_A$ then 
$x^\alpha \in {\mathfrak c}_A$ for any $\alpha \in (0,1)$.  \end{itemize}
\end{lemma}  \begin{proof}     (1) \ Write $x = a + ib$,
for positive $a$ and selfadjoint $b$ in a containing $B(H)$.  By the argument in the proof
of \cite[Lemma 8.1]{BRI},  there exists a selfadjoint $c \in B(H)$
with $b = a^{\frac{1}{2}} c a^{\frac{1}{2}}$ and $\Vert c \Vert \leq \tan \rho$.
Then $x = a^{\frac{1}{2}} (1 + ic) a^{\frac{1}{2}}$, and 
$$x^* x = a^{\frac{1}{2}} (1 + ic)^* a (1 + ic) a^{\frac{1}{2}} \leq
C a.$$
By the $C^*$-identity $\Vert (1 + ic)^* a (1 + ic) \Vert$ equals
$$\Vert a^{\frac{1}{2}} (1 + ic) (1 + ic)^* a^{\frac{1}{2}}  \Vert
\leq \Vert a \Vert (1 + \Vert c \Vert^2) \leq
 \Vert a \Vert (1 + \tan^2 \rho) = \Vert a \Vert  \sec^2 \rho.$$ 
So we can take $C = \Vert a \Vert \, \sec^2 \rho$.  Saying that 
$x^* x \leq
C {\rm Re}(x)$ is the same as saying that $x \in \frac{C}{2} {\mathfrak F}_A$.

(2) \ This follows from (1) since in this case the numerical range of $x^\alpha$ is contained in 
a sector $S_{\rho}$ with
$\rho  < \frac{\pi}{2}$.   \end{proof}  

{\bf Remark.}  The last result is related to the remark before \cite[Lemma 8.1]{BRI}.

\medskip

Of course $\Vert {\rm Im}(x^{\frac{1}{n}}) \Vert \to 0$ as $n \to \infty$, for $x \in {\mathfrak r}_A$ (as is clear e.g.\ from the above).
 
For $x \in {\mathfrak r}_A$, unlike the ${\mathfrak F}_A$ case,
we do not have $||x^{1/m}|| \leq ||x||^{1/m}$.  We are indebted to
Christian Le Merdy for the example
$$\left[ \begin{array}{ccl} 1 & i \\ i & 0 \end{array} \right].$$
Hence if $||x|| \leq 1$ we cannot say that $||x^{1/m}|| \leq 1$ always.
However we have:

\begin{lemma} \label{Bal}  If $||x|| \leq 1$ and $x \in {\mathfrak r}_A$, for an operator algebra $A$, 
 then $||x^{1/m}||
\leq \frac{m^2}{(m-1) \pi} \sin(\frac{\pi}{m}) \leq \frac{m}{m-1}$.
More generally, $||x^{\alpha}|| \leq \frac{\sin(\alpha \pi)}{\pi \alpha (1 - \alpha)}$
if $0 < \alpha < 1$.
Hence for any $\alpha \in (0,1)$ there is a $\delta > 0$ with
$\Vert x \Vert \leq \delta$ implying  $||x^{\alpha} || \leq 1$ for all $x \in {\mathfrak r}_A$.
\end{lemma}

\begin{proof}
This follows from the well known A.V.\ Balakrishnan representation of powers
 $x^\alpha$ as $\frac{\sin(\alpha \pi)}{\pi} \int_0^\infty \, t^{\alpha - 1} \, (t + x)^{-1} x \, dt$
(see e.g.\ \cite{Haase}). 
If we use the simple fact that $\Vert (t + x)^{-1} \Vert \leq \frac{1}{t}$ for accretive operators $x$,
and
$$\Vert (t + x)^{-1} x \Vert = \Vert  (1 + \frac{x}{t})^{-1}  \frac{x}{t} \Vert
= \Vert {\mathfrak F}(\frac{x}{t}) \Vert \leq 1,$$
then 
the norm of $x^\alpha$ is dominated by
$$ \frac{\sin(\alpha \pi)}{\pi} (\int_0^1 t^{\alpha - 1} \, \cdot 1  dt + \int_1^\infty \, t^{\alpha - 1}
\frac{1}{t}  \, dt ) = \frac{\sin(\alpha \pi)}{\pi \alpha (1 - \alpha)}.$$
The rest is clear from this.
\end{proof}

\begin{lemma} \label{strsq}
If $\alpha \in (0,1)$ then there
exists a constant $K$ such that if $a, b \in  {\mathfrak r}_{B(H)}$
for a Hilbert space $H$, and $ab = ba$,  then
 $\Vert (a^{\alpha}  - b^\alpha) \zeta \Vert
\leq K \Vert (a-b) \zeta \Vert^\alpha$, for $\zeta \in H$.
    \end{lemma}

\begin{proof}    By the Balakrishnan representation in the last proof,
 if $\zeta \in {\rm Ball}(H)$ we have
$$(a^{\alpha}  - b^\alpha) \zeta =
\frac{\sin(\alpha \pi)}{\pi} \int_0^\infty \, t^{\alpha - 1} \, [(t + a)^{-1} a - (t + b)^{-1} b] \zeta \, dt.$$
By the inequality $\Vert (t + x)^{-1} \Vert \leq \frac{1}{t}$ for accretive operators $x$, we have
$$\Vert [(t + a)^{-1} a - (t + b)^{-1} b] \zeta \Vert =
\Vert (t + a)^{-1} (t + b)^{-1} (a-b) t \zeta \Vert
\leq \frac{1}{t}  \Vert (a-b) \zeta \Vert ,$$
   and so as in the proof of Lemma \ref{Bal},  $\Vert \int_0^\infty \, t^{\alpha - 1} \, [(t + a)^{-1} a - (t + b)^{-1} b] \zeta \, dt \Vert$ is dominated
by $$2  \int_0^\delta \, t^{\alpha - 1}  \, dt  + \int_{\delta}^\infty \,  t^{\alpha - 2} \, dt \, \Vert (a-b) \zeta \Vert
= \frac{2}{\alpha} \delta^{\alpha} + \frac{\delta^{\alpha -1}}{1-\alpha} \, \Vert (a-b) \zeta \Vert$$
for any $\delta > 0$.
We may now set $\delta = \Vert (a-b) \zeta \Vert$ to obtain our inequality.
   \end{proof}

{\bf Remark.}   The proof above uses a trick from \cite[Theorem 1]{MP}, and perhaps can be extended
to the class of operators considered there.

\begin{lemma} \label{vpow}  If $a \in {\mathfrak r}_A$ for an operator algebra $A$,
and $v$ is a partial isometry in any
 containing $C^*$-algebra $B$ with $v^* v = s(a)$, then
$v a v^* \in {\mathfrak r}_B$ and
$(v a v^*)^r = v a^r v^*$ if $r \in (0,1) \cup \Ndb$.
\end{lemma}

\begin{proof}   This is clear if $r  = k \in \Ndb$.   It is also clear that $v a v^* \in {\mathfrak r}_B$.
We will use the Balakrishnan representation above to check that $(v a v^*)^r = v a^r v^*$ if $r \in (0,1)$ (it can also 
be deduced from the ${\mathfrak F}_A$ case in \cite{BNII}). 
Claim: $(t + v a v^*)^{-1} v a v^* = v (t + a)^{-1} a v^*$.    Indeed since $v^* v a = a$ we have 
$$(t + v a v^*) v (t + a)^{-1} a v^* 
= v (t+a) (t + a)^{-1} a v^* = v a v^*,$$
proving the Claim.  Hence for any $\zeta, \eta \in H$
we have
$$\langle (t + v a v^*)^{-1} v a v^* \zeta, \eta \rangle
= \langle  v (t + a)^{-1} a v^*  \zeta, \eta \rangle 
= \langle  (t + a)^{-1} a v^*  \zeta, v^* \eta \rangle.$$   
Hence by the Balakrishnan representation $\langle (v a v^*)^r \zeta, \eta \rangle$ equals
$$\frac{\sin(r \pi)}{\pi} \int_0^\infty \, t^{r - 1} \, \langle
 (t + vav^*)^{-1} v a v^* \zeta , \eta \rangle \, dt = 
\frac{\sin(r \pi)}{\pi} \int_0^\infty \, t^{r - 1} \, \langle  (t + a)^{-1} a v^*  \zeta, v^* \eta \rangle \, dt,$$
which equals $\langle v a^r v^* \zeta, \eta \rangle$, as desired.
\end{proof}

The last result generalizes \cite[Lemma 1.4]{BNI}.   With the last few results  in hand, particularly 
Lemma \ref{Bal} and Lemma \ref{vpow}, it appears that 
all of the results in \cite{BNI} stated in terms of  ${\mathfrak F}_A$ (or $\frac{1}{2}  {\mathfrak F}_A$ or ${\mathfrak c}_A$),
should generalize without problem to the ${\mathfrak r}_A$ case.  We admit that we have not yet carefully checked every part of every result in 
\cite{BNI}  
  for this though, but hope to
in forthcoming work.     

\subsection{Commuting operators in ${\mathfrak r}_A$}   If $S \subset {\mathfrak r}_A$, for an operator algebra
$A$,  and if $xy = yx$
for all $x, y \in S$, write oa$(S)$ for the smallest closed subalgebra of $A$ containing $S$.

\begin{lemma} \label{commt}  If $S$ is a commuting subset of ${\mathfrak r}_A$
then ${\rm oa}(S)$ has a cai.
\end{lemma}

\begin{proof}  Let $C = {\rm oa}(S)$.  Then $C$ contains oa$(x)$ for each $x \in S$,
and hence $C^{\perp \perp}$ contains $s(x)$.
Thus $C^{\perp \perp}$ contains  $p = \vee_{x \in S} \, s(x))$ 
(by the comments on `meets' and `joins' on p.\ 190 of \cite{BRI}).  Now it is 
easy to see that $p$ is an identity for $C^{\perp \perp}$, so that 
$C$  has a cai (by e.g.\ \cite[Proposition 2.5.8]{BLM}).  \end{proof} 

If $x, y$ are commuting elements of $\frac{1}{2} {\mathfrak F}_A$  or  ${\mathfrak r}_A$, for an operator algebra $A$, 
then it is not necessarily true   
that $xy$ has numerical range excluding the negative real axis.  (Indeed this can be false
 even in $M_2$, and even if the `imaginary parts' of $x$ and 
$y$ have tiny norms compared to their `real parts', so that their numerical range
consists of numbers in the unit disk with argument very close to zero.  Here $x$ and $y$ are very close to 
being positive). Thus 
one cannot define $(xy)^{\frac{1}{2}}$ as in e.g.\ \cite{LRS}; rather in such cases one should 
define $(xy)^{\frac{1}{2}}$ to be 
$x^{\frac{1}{2}} y^{\frac{1}{2}}$.
It is proved in \cite{BBS} that $(xy)^{\frac{1}{2}}$ is again in $\frac{1}{2} {\mathfrak F}_A$.   
More generally, the product of the $n$th roots of $n$ mutually commuting members of $\frac{1}{2} {\mathfrak F}_A$ (resp.\
${\mathfrak r}_A$)
 is again in $\frac{1}{2} {\mathfrak F}_A$ (resp.\ in 
${\mathfrak r}_A$).   This fact has a nice application
in Section \ref{pury} below.

\subsection{Concavity, monotonicity, and operator inequalities}  The usual operator
concavity/convexity results for $C^*$-algebras seem to fail for the
${\mathfrak r}$-ordering.  That is, results of the type in
\cite[Proposition 1.3.11]{Ped} and its proof fail.
Indeed, functions like  Re$(z^{\frac{1}{2}}),
{\rm Re}(z (1+z)^{-1}), {\rm Re}(z^{-1})$
are not operator concave or convex,  even for
operators  $x, y \in \frac{1}{2} {\mathfrak F}_A$.  In
fact this fails even in the simplest case $A = \Cdb$,
taking $x = \frac{1}{2}, y = \frac{1+i}{2}$.  Similar remarks hold for `operator monotonicity'
with respect to the ${\mathfrak r}_A$-ordering for these functions.

For   the 
${\mathfrak r}$-ordering, a   way one can often prove operator inequalities, or that something is increasing, is as follows.  Suppose for example that $f, g$ are functions in the disk algebra,
 with Re$(f(z)) \leq {\rm Re}(g(z))$ for all $z \in \Cdb$ with $|z| \leq 1$.
If $x \in \frac{1}{2} {\mathfrak F}_A$, for an  operator algebra $A$, then we can deduce
that Re$(f(1-2x)) \leq  {\rm Re}(g(1-2x))$.  Here e.g.\ $f(1-2x)$  is the `disk algebra functional calculus', arising from von Neumann's inequality for the contraction $1-2x$.   The reason for this is that ${\rm Re}(g(z) - f(z)) \geq 0$ on the disk,
and so by  \cite[Proposition 3.1, Chapter IV]{NF}  we have that ${\rm Re}(g(1-2x) - f(1-2x)) \geq 0.$

As an illustration of this principle, it follows easily by this idea that for any  $x \in \frac{1}{2} {\mathfrak F}_A$, the sequence
$({\rm  Re}(x^{\frac{1}{n}}))$ is increasing (see \cite{BBS}).   The last fact is another example of  $\frac{1}{2}   {\mathfrak F}_A$ behaving better than
${\mathfrak r}_A$:  for contractions $x \in {\mathfrak r}_A$, we do not in
general have $({\rm Re} (x^{1/m}))$ increasing  with $m$.
The matrix example just above Lemma \ref{Bal} will demonstrate this.   However one can show that
for any $x \in {\mathfrak r}_A$ there exists a constant $c > 0$ such that 
$({\rm Re} ((x/c)^{1/m}))_{m \geq 2}$ is  increasing  with $m$.   Indeed if 
$c = (2 \Vert {\rm Re}(x^{\frac{1}{2}}) \Vert)^2$, then by Lemma \ref{rootf} (2) we have 
$(x/c)^{\frac{1}{2}} \in \frac{1}{2} {\mathfrak F}_A$.
Thus ${\rm Re} ((x/c)^{t})$ increases as $t \searrow 0$ 
(see the proof of the \cite[Proposition 3.4]{BBS}),
from which the desired assertion  follows.

\begin{corollary} \label{incai}  If  
$A$ is a separable  approximately unital
operator algebra,  then $A$ has a commuting cai which is increasing in the 
${\mathfrak r}$-ordering.   \end{corollary}  \begin{proof}  As in \cite[Corollary 2.18]{BRI}, 
there is a  $x \in \frac{1}{2} {\mathfrak F}_A$ such that 
$(x^{1/n})$ is a  commuting cai.  The result now follows by the fact in the last paragraph.
\end{proof}

We do not know if this is true in the nonseparable case.  See also the  remarks after Corollary \ref{dirset}.

Finally, we clarify a few imprecisions in a couple 
 of the positivity results in \cite{BRI, BRII}.   At the end of Section 4
of \cite{BRII}, states on a nonunital algebra should probably also be assumed to have norm 1  (although the arguments there do not need this). 
 In \cite[Proposition 4.3]{BRI} we should have explicitly stated the hypothesis that $A$ is
approximately unital.  There are some small typo's in the proof of \cite[Theorem 2.12]{BRI}
but the reader should have no problem correcting these.

\section{Strictly real positive elements} \label{morp2} 

An element in $A$ with
${\rm Re}(\varphi(x)) > 0$ for all states on $A^1$
whose restriction 
to $A$ is nonzero, will be called 
{\em strictly real positive}.   Such $x$ are in $ {\mathfrak r}_A$.
This includes the $x \in A$ with Re$(x)$ strictly positive in some 
$C^*$-algebra generated by $A$.  
  If $A$ is approximately unital, then these conditions are in fact   equivalent, 
as the next result shows.
Thus the definition of  strictly real positive here generalizes the  definition 
given in \cite{BRI} for approximately unital operator algebras.

\begin{lemma} \label{secto}   Let $A$ be an approximately unital operator algebra, which 
generates a  $C^*$-algebra $C^*(A)$. 
An element $x \in A$ is strictly real positive in the sense above 
iff ${\rm Re}(x)$ is strictly positive in  
$C^*(A)$.   \end{lemma}  \begin{proof}  The one direction 
 follows because any state 
on $A^1$ whose restriction 
to $A$ is nonzero, extends to a state on $C^*(A)^1$ which is nonzero on 
$C^*(A)$.  The restriction to $C^*(A)$ of the latter state is a positive multiple of a state.

For the other direction 
recall that we showed in the introduction that  
any state on  $C^*(A)$ gives rise to a state on $A^1$.
Since any cai of $A$ is a cai of $C^*(A)$, the latter state cannot vanish on $A$.
  \end{proof}


{\bf Remark.}  Note that if ${\rm Re}(x) \geq \epsilon 1$ in $C^*(A)^1$, then there exists
a constant $C > 0$ with ${\rm Re}(x) \geq \epsilon 1 \geq C x^* x$,
and it follows that $x \in \Rdb_+ {\mathfrak F}_A$.  
Thus if $A$ is unital then every
strictly real positive in $A$ is 
 in $\Rdb_+ {\mathfrak F}_A$.   However this is false 
if $A$ is approximately unital (it is even easily seen to be 
false in the $C^*$-algebra  $A = c_0$).  Conversely, note that 
if $A$ is an approximately unital operator algebra with no 
r-ideals and no identity, then every nonzero element of 
$\Rdb_+ {\mathfrak F}_A$ is strictly real positive
by \cite{BRI} Theorem 4.1.

We also remark that it is tempting to define 
an element  $x \in A$ to be strictly real positive if Re$(x)$ strictly positive in some
$C^*$-algebra generated by $A$.  However this definition 
can depend on the particular generated $C^*$-algebra, unless one only uses
states on the latter that are not allowed to vanish on $A$ (in which
case it is equivalent to other definition).   As an example of this, consider the 
algebra of $2 \times 2$ matrices supported on the first row, and the various $C^*$-algebras it
can generate.

\medskip

We next discuss how results  in \cite{BRI}  generalize, particularly those 
related to strict real positivity if 
we use the definition at the start of the present section.
  We will need some of this in  Section \ref{fgsect}  below.  
We recall that in \cite{BRI}, many `positivity' results were established for elements in ${\mathfrak F}_A$ or $\frac{1}{2} {\mathfrak F}_A$,
and by extension for the proper cone ${\mathfrak c}_A = 
\Rdb_+ {\mathfrak F}_A$.  In  \cite[Section 3]{BRII} we pointed out several of these facts
that generalized to the larger cone ${\mathfrak r}_A$, and indicated that some of this would be discussed in more detail in
\cite{BBS}.  In  \cite[Section 4]{BRII} we pointed out that the hypothesis in many of these results that 
$A$ be approximately unital could be simultaneously relaxed. 
In the next 
few paragraphs we give a few more details, that indicate the 
similarities and differences between these cones, particularly focusing on the results involving 
strictly real positive elements.    The following list should be added to the list in \cite[Section 3]{BRII},
and some complementary details are discussed in \cite{BBS}.  

In \cite[Lemma 2.9]{BRI} the ($\Leftarrow$) direction
is correct  for $x \in {\mathfrak r}_A$ with the same proof.   Also one need not assume there that 
$A$ is approximately unital, as we said towards the end of Section 4 in 
\cite{BRII}.    
The other direction is not true in general (not even in  $A = \ell^\infty_2$, see example in \cite{BBS}),
but there is a partial result, Lemma \ref{sect} below.

In \cite[Lemma 2.10]{BRI},
 (v) implies (iv) implies (iii) (or equivalently (i) or (ii)),  with ${\mathfrak r}_A$
in place of ${\mathfrak F}_A$, using the ${\mathfrak r}_A$ version above of the ($\Leftarrow$) direction of 
\cite[Lemma 2.9]{BRI}, and \cite[Theorem 3.2]{BRII} (which 
gives $s(x) = s({\mathfrak F}(x))$).
 However none of the other implications in that lemma are correct, even in $\ell^\infty_2$.

Proposition 2.11 and Theorem 2.19 of \cite{BRI} are correct in their ${\mathfrak r}_A$ variant, which should 
be phrased in terms of 
 strictly real positive elements in ${\mathfrak r}_A$ as defined above at the start of
the present section.  
Indeed this variant of Proposition 2.11 is true even for nonunital algebras 
if  in the proof we replace $C^*(A)$ by $A^1$.  
Theorem 2.19  of \cite{BRI} may be 
seen using the parts of \cite[Lemma 2.10]{BRI} which are true for ${\mathfrak r}_A$
in place of ${\mathfrak F}_A$, and \cite[Theorem 3.2]{BRII} (which 
gives $s(x) = s({\mathfrak F}(x))$).    Lemma 2.14 of \cite{BRI} 
is  clearly false even in $\Cdb$, however
it is true with essentially the same proof if the elements $x_k$ there
are strictly real positive elements, or more generally if they are in ${\mathfrak r}_A$ and their  numerical ranges in $A^1$ intersects the imaginary 
axis only possibly at $0$.  
 Also,  this 
does not effect the correctness of the important results that follow it in \cite[Section 2]{BRI}.  Indeed as  stated 
in \cite{BRII},  all descriptions of r-ideals and $\ell$-ideals and HSA's from \cite{BRI} are valid 
with ${\mathfrak r}_A$ in place of ${\mathfrak F}_A$, sometimes by using \cite[Corollaries 3.4 and 3.5]{BRII}).  We remark that Proposition 2.22 of  \cite{BRI} is clearly false
with ${\mathfrak F}_A$ replaced by ${\mathfrak r}_A$, even in $\Cdb$.
  
Similarly, in  \cite{BRI} Theorem 4.1, (c) implies (a) and (b) there with ${\mathfrak r}_A$
in place of ${\mathfrak F}_A$.   However the Volterra algebra \cite[Example 4.3]{BRI}
is an example where (a) in \cite{BRI} Theorem 4.1 holds but not (c) (note that 
the Volterra operator $V \in {\mathfrak r}_A$, but $V$ is not strictly real positive in 
$A$).   The results in Section 3 of  \cite{BRI}  will be discussed later in Section
\ref{fgsect}.

\begin{lemma} \label{sect}  In an operator algebra $A$, 
suppose that $x \in {\mathfrak r}_A$ and either $x$ is strictly real positive, or 
the numerical range $w(x)$ of $x$ in $A^1$ is contained in a sector 
$S_\psi$ of angle $\psi < \pi/2$ (see notation above Lemma {\rm \ref{roots}}).  If $\varphi$ is a state 
on $A$ or more generally on $A^1$, 
then $\varphi(s(x)) = 0$ iff 
$\varphi(x) = 0$.   \end{lemma}  \begin{proof}   The one direction is as in 
\cite[Lemma 2.9]{BRI}  as mentioned above.  The strictly real positive case of the 
other direction is obvious (but non-vacuous in the $A^1$ case).  In the remaining case, 
 write $\varphi = \langle \pi(\cdot) \xi , \xi \rangle$
for a unital $*$-representation  $\pi$ of $C^*(A^1)$ on a Hilbert space $H$,
and a unit vector $\xi \in H$.
Then $w(\pi(x))$ is contained in a sector of the same angle.
By Lemma 5.3 in Chapter IV of \cite{NF} we have $\Vert \pi(x) \xi \Vert^2
= \varphi(x^* x) = 0$.  As e.g.\ in the proof of 
\cite[Lemma 2.9]{BRI} this gives $\varphi(s(x)) = 0$.
  \end{proof}

\begin{corollary} \label{sx}  Let $x \in {\mathfrak r}_A$ for  an operator algebra $A$.  If  $\varphi(x^{\frac{1}{n}}) = 0$
for some $n \in \Ndb, n \geq 2$, and state  $\varphi$ on $A$,
then $\varphi(s(x)) = 0$ and $\varphi(x^{\frac{1}{m}}) = 0$
for all $m \in \Ndb$.  Thus if $\varphi(s(x)) \neq 0$ for a state  $\varphi$ on $A$, 
then ${\rm Re}(\varphi(x^{\frac{1}{n}})) > 0$ 
for all $n \in \Ndb, n \geq 2$.
   \end{corollary}  \begin{proof}
It is clear that $s(x) = s(x^{\frac{1}{m}})$ for all $m \in \Ndb$, by using for example the fact from
\cite[Section 3]{BRII} that $x^{\frac{1}{n}} \to s(x)$ weak*.  
Since the numerical range of $x^{\frac{1}{n}}$ 
in $A^1$ is contained in a sector
centered on the positive real axis of angle $< \pi$,
$\varphi(s(x)) = \varphi(s(x^{\frac{1}{n}})) = 0$ by Lemma \ref{sect}.
As we said above, this implies that 
$\varphi(x)  = 0$, and the same argument applies with
$x$ replaced by $x^{\frac{1}{m}}$ to give $\varphi(x^{\frac{1}{m}}) = 0$.   

The last statement follows from this, since ${\rm Re}(\varphi(x^{\frac{1}{n}})) > 0$
is equivalent to $\varphi(x^{\frac{1}{n}}) \neq 0$ if $n \geq 2$.  \end{proof}

{\bf Remark.}   Examining the proofs of the last three results show that they are 
valid if states on $A$ are replaced by  nonzero functionals that extend to states on 
$A^1$, or equivalently extend to  a $C^*$-algebra generated by $A^1$.

\begin{corollary} \label{sx3}  
In an operator algebra $A$,
if $x \in {\mathfrak r}_A$ and $x$ is strictly real positive,
then $x^{\frac{1}{n}}$ is strictly real positive for all $n \in \Ndb$. 
 \end{corollary}

\begin{proof}
If $x^{\frac{1}{n}}$ is not strictly real positive for some $n \geq 2$,
then $\varphi(x^{\frac{1}{n}}) = 0$ for some state 
$\varphi$ of $A^1$ which is nonzero on $A$.  Such a state extends to a state 
on $C^*(A^1)$.  By the last Remark,  $\varphi(x)  = 0$ by Corollary \ref{sx}, a contradiction.
\end{proof}

\section{Principal $r$-ideals} \label{fgsect}

Section 3 in our earlier paper \cite{BRI} was entitled
``When $xA$ and $Ax$ are closed".  We showed there that if $x \in {\mathfrak F}_A$ then  
both $xA$ and $Ax$ are closed iff $x$ is pseudo-invertible (that is, 
there exists $y \in A$  with $xy x = x$).   We can improve this result as follows:

\begin{theorem}   \label{ws}  For
an operator algebra $A$, if $x \in {\mathfrak r}_A$, then the following
are equivalent:
  \begin{itemize}
\item [(i)] $s(x) \in A$,
\item [(ii)]  $xA$ is closed,
\item [(iii)] $Ax$  is closed,
\item [(iv)]  $x$ is pseudo-invertible in $A$,
 \item [(v)]  $x$ is invertible in ${\rm oa}(x)$. 
\end{itemize}
Moreover, all of Theorem {\rm 3.2} of \cite{BRI} is valid for 
$x \in {\mathfrak r}_A$.   
\end{theorem}

 \begin{proof}  
We recall that $(x^{\frac{1}{m}})_{m \in \Ndb}$ is a bai for ${\rm oa}(x)$, by \cite[Theorem 3.1]{BRII},
and it has weak* limit $s(x) \in {\rm oa}(x)^{\perp \perp} \subset A^{**}$.

(ii)  $\Rightarrow$ (i) \ Suppose $xA$ is closed.  Then   $$x^{\frac{1}{2}} \in {\rm oa}(x) \subset \overline{x  {\rm oa}(x)} \subset 
\overline{xA} = xA,$$
so $x^{\frac{1}{2}} = xy$ for some $y \in A$.
Thus if $z = x^{\frac{1}{2}} y \in A$  then $x =  x^{\frac{1}{2}} x y = x z$,
and so $a = a z$ for every $a \in {\rm oa}(x)$.
Now  $s(x) z = z$ since $x^{\frac{1}{2}} \in {\rm oa}(x)$ for example.
On the other hand $s(x) z = s(x)$ since $s(x)$ is a weak* limit of the 
bai in ${\rm oa}(x)$.     Thus $s(x) = z \in A$.

(i) $\Rightarrow$ (iv) \ Since $s(x)$ is the identity of ${\rm oa}(x)^{**}$, 
(i) is equivalent to ${\rm oa}(x)$ being unital.  This implies by the Neumann
lemma that $x$ is invertible in ${\rm oa}(x)$, hence that $x$ is pseudo-invertible in $A$.

(iv) $\Rightarrow$ (ii) \ (iv) implies
that $xA = xy A$ is closed since $xy$ is idempotent.   

That (iii) is equivalent to the others follows from (ii) and the 
symmetry in (i) or (iv).   That (v) is equivalent to (i) is as in 
\cite[Theorem 3.2]{BRI}.    

The rest follows  almost identically to  \cite[Theorem 3.2]{BRI}, but
using the next Proposition  in place of \cite[Theorem 2.12]{BRI} (in the first
line on page 204 one needs to replace the disk by the right half plane).  
\end{proof}

The next result, used in the last paragraph of the last proof,
 is an analogue of \cite[Theorem 2.12]{BRI}: 

\begin{proposition} \label{okins} 
If $A$ is a
 subalgebra of $B(H)$,
and $x \in {\mathfrak r}_A$,  then $x$ is invertible in $B(H)$ iff $x$ 
is invertible in $A$, and iff  ${\rm oa}(x)$ contains $I_H$; and in this case $s(x) = I_H$.
\end{proposition} \begin{proof}  It is clear by the Neumann
lemma that if ${\rm oa}(x)$ contains $I_H$ then $x$ is invertible in 
${\rm oa}(x)$, and hence in $A$.
Conversely, if $x$ is invertible in $B(H)$ then by the equivalences (i)--(iv)
proved in the last theorem, 
with $A$ replaced by $B(H)$ we have $s(x) \in B(H)$, and this is the identity
of ${\rm oa}(x)$.  If $x y = I_H$ in $B(H)$, 
then $I_H = xy = s(x) x y = s(x)$.  The latter implies that $I_H \in {\rm oa}(x) \subset
A$.
\end{proof}

\begin{corollary} \label{Aha3}   Let $A$ be an operator algebra.  A closed right ideal $J$ of $A$
is of the form $x A$ for some $x \in {\mathfrak r}_A$  iff  
$J  = eA$ for a projection $e \in A$.  
\end{corollary}   \begin{proof}   If $xA$ is closed for a nonzero $x \in {\mathfrak r}_A$
then by the theorem $e = s(x) \in A$.   Hence  $xA = eA$ by \cite[Corollary 3.5]{BRII} .
The other direction is trivial.   \end{proof}

\begin{corollary} \label{Aha2}   If a  nonunital operator algebra
$A$ contains a nonzero $x \in {\mathfrak r}_A$ with $xA$ closed,
 then $A$ contains  a nontrivial projection.  If also $A$ is approximately unital 
then this implies that $A$ has a nontrivial
r-ideal.  \end{corollary}   \begin{proof}  
By the above $xA = eA$ for a projection $e \in A$.  Now $e \neq 0$ since $x \in \overline{x  {\rm oa}(x)} \subset  xA = eA$.
If $eA = A$ then $e$ is a left identity for $A$, hence is a 
two-sided identity if $A$ is approximately unital
(since $e e_t = e_t \to e$ for the cai $(e_t)$).
This contradiction shows that $eA$ is a nontrivial r-ideal.  
\end{proof}

\begin{corollary} \label{Aha}   If an operator algebra $A$ has a cai but no 
identity, then $x A  \neq A$ for all $x \in {\mathfrak r}_A$.  
\end{corollary}   \begin{proof}
 If $xA  = A$ then the last proof   shows that $A$ is unital. \end{proof}

It follows from this, as in \cite{BRI}, that if $x$ is a strictly real positive element (in new our sense above) in
a nonunital approximately unital operator algebra $A$, then $xA$ is not closed.

\begin{corollary} \label{Aha4}   Let $A$ be an operator algebra.   A closed $r$-ideal $J$ in $A$ 
is algebraically finitely generated as a right module over $A$   iff   $J  = eA$ for a projection $e \in A$.  This
is also equivalent to $A$ being algebraically finitely generated as a right module over $A^1$.
\end{corollary}   \begin{proof}  
Let $J$ be an r-ideal which is algebraically finitely generated over $A$ by elements $x_1, \cdots, x_n \in A$.
By \cite[Theorem 2.15]{BRI}, $J$ is the closure of $\cup_t \, J_t$ of an increasing net 
of r-ideals $J_t = \overline{a_t A}$.  By \cite[Lemma 1]{ST}, $J = \cup_t \, J_t$.   It follows that
for one of these $t$ we have $x_k \in J_t$ for all $k = 1, \cdots, n$, and so $J = J_t$.
By \cite[Lemma 1]{ST} again, $J = a_t A$.   By Corollary \ref{Aha3}, $J = eA$. 

If  $J$ is algebraically finitely generated  over $A^1$  then by the above 
$J  = eA^1$.  Clearly $e \in A$, and so $J = \{ x \in A : e x = x \} = eA$.  \end{proof}


{\bf Remark.}   One of our few remaining open questions concerning the ${\mathfrak r}$-ordering,
is if $q({\mathfrak r}_A) = {\mathfrak r}_{A/I}$, where $q$ is the 
canonical quotient map from an approximately unital operator algebra $A$ 
onto its quotient by an approximately unital ideal $I$.   Nor are we sure if the `strictly real positive' variant of this
is true if $A$ is approximately unital but 
nonunital.  It is easy to see that this is all true if the support projection of $I$ is in $M(A)$.  We showed 
in \cite[Section 6]{BRI} that $q({\mathfrak F}_A) = {\mathfrak F}_{A/J}$.  These may be regarded 
as analogues of the `lifting of positive or selfadjoint elements
in $C^*$-algebras.  They may also  be regarded
as peak interpolation results (see \cite{Bnew} for a brief survey of
peak interpolation).  For example, that $q({\mathfrak F}_A) = {\mathfrak F}_{A/J}$ is saying in
the case of a unital function algebra on a compact set $K$ that given a function $f \in A$ 
with $|1-2f(x)| \leq 1$ on a $p$-set $E$ in $K$, there exists
a $g \in A$ with $|1-2g(x)| \leq 1$ on all of $K$ and $g = f$ on $E$.     

\section{Positivity in the Urysohn lemma}  \label{pury}  

In our previous work \cite{BHN,BRI, BNII, BRII} we had two main settings for
 noncommutative Urysohn lemmata for a subalgebra $A$ of a $C^*$-algebra $B$.  In both
settings we have a compact projection $q \in A^{**}$, dominated
by an open projection $u$ in $B^{**}$, and we seek to find $a \in {\rm Ball}(A)$ with
$aq = q a = q$, and both $a \, u^{\perp}$ and $u^{\perp} \, a$ either small or zero.
In the first setting $u \in A^{**}$ too, whereas this is not required
in the second setting.  We now ask if in both settings one may also
have $a$ very close to a positive operator (in the usual sense),
and $a$ `positive' in our new sense (involving ${\mathfrak r}_A$ or
sets related to ${\mathfrak F}_A$)?  In the first setting, all works
perfectly:

\begin{theorem} \label{urysII}  Let $A$ be an
operator algebra (not necessarily approximately unital), and let
$q \in A^{**}$ be a compact projection, which
is dominated by an open projection  $u \in A^{**}$.  Then
there exists an $a \in \frac{1}{2} {\mathfrak F}_A$ with $a q = qa = q$,
and $a u = ua = a$.  Indeed for any
$\epsilon > 0$ this can be done with in addition
the numerical range (and spectrum) of $a$ within
a horizontal cigar centered on the line segment $[0,1]$ in the
$x$-axis, of height $< \epsilon$.  Such $a$ is within distance
$\epsilon$ of a positive operator in $A + A^*$.
\end{theorem}
 \begin{proof}  The proofs of \cite[Theorem 2.6]{BNII} and \cite[Theorem 6.6 (2)]{BRII}
show that this all can be done with $a \in \frac{1}{2} {\mathfrak F}_A$.
 Then $a^{\frac{1}{n}} q = q a^{\frac{1}{n}}  = q$, as is clear for example
using the power series form $a^{\frac{1}{n}} = \sum_{k=0}^\infty
\, { 1/n \choose k} \, (-1)^k (1-a)^k$ from \cite[Section 2]{BRI},
where it is also shown that $a^{\frac{1}{n}} \in \frac{1}{2} {\mathfrak F}_A$.
Similarly $a^{\frac{1}{n}} u = u a^{\frac{1}{n}} = a^{\frac{1}{n}}$, since $u$ is
the identity multiplier on oa$(a)$ which contains these roots \cite[Section 2]{BRI}.
That the numerical range of $a^{\frac{1}{n}}$ lies in
the desired cigar is as in the proof of \cite[Theorem 2.4]{BRI}.
The distance
$\epsilon$ assertion is explained in  \cite{BRI}
after the just cited result, indeed 
the element  is within distance
$\epsilon$ of its real part, which is positive.  \end{proof}

We now turn to the second
setting  (see e.g.\ \cite[Theorem 6.6 (1)]{BRII}), where the dominating open projection
$u$ is not required
to be in $A^{\perp \perp}$.
Of course if $A$ has no identity or cai then one cannot expect
the `interpolating' element $a$ to be in $\frac{1}{2} {\mathfrak F}_A$ or
${\mathfrak r}_A$.  This  may  be seen clearly in the case that $A$ is the functions in the disk algebra vanishing
at $0$.  Here $\frac{1}{2} {\mathfrak F}_A$ and 
${\mathfrak r}_A$ are $(0)$.  Indeed by the maximum modulus theorem for harmonic functions there are  
no nonconstant 
functions in this algebra which have nonnegative real part.
The remaining question is the approximately unital case `with positivity'.
We solve this next, also solving the questions posed at the end of \cite{BNII}.

\begin{theorem} \label{urys}  Let $A$ be an approximately unital
subalgebra of a  $C^*$-algebra $B$, and let
$q \in A^{\perp \perp}$ be a compact projection.
\begin{itemize}
\item [(1)]  If $q$ is dominated by an open projection  $u \in B^{**}$.  For any $\epsilon > 0$,
there exists an $a \in 
\frac{1}{2} {\mathfrak F}_A$ with $a q = q$,
and $\Vert a (1-u) \Vert < \epsilon$ and $\Vert (1-u)  a \Vert < \epsilon$.
Indeed this can be done with in addition
the numerical range (and spectrum) of $a$ within
a horizontal cigar centered on the line segment $[0,1]$ 
in the
$x$-axis, of height $< \epsilon$.  Again, such $a$ is within
distance $\epsilon$ of a positive operator in $A + A^*$. 
\item [(2)]  $q$ is a weak* limit of some net $(y_t) \in \frac{1}{2}
\, {\mathfrak F}_A$ 
with $y_t q = q y_t = q$.
\end{itemize} \end{theorem}
 \begin{proof}    (2) \  First assume that $q = u(x)$ for some $x \in \frac{1}{2} {\mathfrak F}_A$.
We may replace $A$ by the commutative algebra ${\rm oa}(x)$, and then $q$ is a minimal projection,
 since $q \, p(x) \in \Cdb q$ for any polynomial $p$.  
Now $q$ is closed and compact in $(A^{1})^{**}$, so the unital case of (2), which follows from 
\cite[Theorem 2.24]{BRI} and the closing remarks to \cite{BNII}, there is a net 
$(z_t) \in \frac{1}{2}
\, {\mathfrak F}_{A^1}$ with $z_t q = q z_t = q$ and $z_t \to q$ weak*.  Let 
$y_t  = z_t^{\frac{1}{2}} x^{\frac{1}{2}}$.  By the paragraph after Lemma \ref{commt} (namely by 
a result from \cite{BBS} referenced there), 
 we have $y_t 
\in \frac{1}{2} {\mathfrak F}_{A^1}
\cap A = \frac{1}{2} {\mathfrak F}_{A}$.  Also, $x^{\frac{1}{2}} q = q x^{\frac{1}{2}} = q$  by considerations
used in the last proof, and similarly $z_t^{\frac{1}{2}} q = q z_t^{\frac{1}{2}} = q$.
Thus $y_t^{\frac{1}{2}} q = q y_t^{\frac{1}{2}} = q$.   If $A$ is represented nondegenerately on a Hilbert
space $H$, and we identify $1_{A^1}$ with $I_H$, then
 for any $\zeta \in H$ we have by Lemma \ref{strsq} that 
$$\Vert (y_t - q) \zeta \Vert
= \Vert (z_t^{\frac{1}{2}} - q) x^{\frac{1}{2}} \zeta \Vert
\leq K \Vert (z_t - q) x^{\frac{1}{2}} \zeta \Vert^{\frac{1}{2}} \to 0 .$$
Thus $y_t \to q$ strongly and hence weak*.

(1) \ If $A$ is unital, the first assertion of (1) 
is \cite[Theorem 2.24]{BRI}.   In the approximately unital case,
by the ideas in the closing remarks to \cite{BNII},
the first assertion of (1) should be equivalent to (2).   Indeed, by feeding such a net $(y_t)$ into
the proof of \cite[Theorem 2.1]{BNII} one obtains the
first assertion of (1).

Next, for an arbitrary compact projection $q \in  A^{\perp \perp}$, by \cite[Theorem 3.4]{BNII}
there exists a net $x_s \in \frac{1}{2} {\mathfrak F}_A$  with $u(x_s) \searrow q$.
By the last paragraph there exist nets $y_t^s \in \frac{1+\epsilon}{2} \, {\mathfrak F}_A$
with $y_t^s \, u(x_s) = u(x_s) \, y_t^s = u(x_s)$, and $y_t^s \to u(x_s)$ weak*.
Then $$y_t^s q = y_t^s \, u(x_s) \, q = u(x_s) \, q = q,$$ for each $t, s$.  It is clear that
the $y_t^s$ can be arranged into a net weak* convergent to $q$.
This proves the Claim.

Finally, we obtain the `cigar' assertion.
For $(y_t)$ as in our Claim, we feed the net $(y_t^{\frac{1}{m}})$ into
the proof of \cite[Theorem 2.1]{BNII}.  Here $m$ is a fixed integer so large that the
numerical range of $y_t^{\frac{1}{m}}$ lies within the appropriate  horizontal cigar,
as in the proof of the previous theorem.  As in that
proof $y_t^{\frac{1}{m}} q = q y_t^{\frac{1}{m}} = q$ and
$y_t^{\frac{1}{m}} \to q$ weak*.    The distance
assertion also follows as in the previous proof.  \end{proof}

{\bf Remarks.}  1)   \ The  recent paper  \cite{CGK} 
contains a special kind of  `Urysohn lemma with positivity' for function algebras.  Some of the conditions of our Urysohn lemma are  
more general than theirs; also our interpolating elements have range in a cigar in the right half plane, as opposed to their Stolz region which 
contains $0$ as an interior point.  Hopefully our results  could be helpful
in such applications.  

2) \ Part (2) of the theorem easily gives a characterization of compact projections $q$  in 
approximately unital operator algebras, as weak* limits of nets $(y_t)$ in $A$, where
$y_t q = q$ and $y_t \in  \frac{1}{2} {\mathfrak F}_A$, and $y_t$ is `as close to being positive as
one likes'.

\section{A semisimple operator algebra which is a modular annihilator algebra but is not weakly compact}
  \label{notw}

  In \cite[p.\ 76]{ABR} we asked if every approximately unital semisimple operator algebra which is a modular annihilator algebra, is weakly compact, or is nc-discrete.     We recall that $A$ is  {\em nc-discrete} if all the open projections in $A^{**}$ are
also closed (or equivalently lie in the multiplier algebra $M(A)$).    In this section 
we will construct an interesting operator algebra $A$ which answers these questions in the negative.  

Let $(c_n)$ be an unbounded increasing sequence in $(0,\infty)$.  For each $n \in \Ndb$ let 
$d_n$ be the diagonal matrix in $M_n$ with 
$c_n^k$ as the $k$th  diagonal entry.  If $M$ is the von Neumann algebra $\oplus_n^\infty \, (M_n  \oplus
M_n)$, we let $N$ be its weak*-closed unital subalgebra consisting of tuples $((x_n , d_n x_n d_n^{-1}))$,
for all $(x_n) \in \oplus_n^\infty \, M_n$.       We define $A_{00}$ to
be the finitely supported tuples in $N$, and  $A_0$ to be the closure of $A_{00}$.
 That is, $A_0$ is the intersection of the 
$c_0$-sum $C^*$-algebra  $\oplus_n^{\circ} \, (M_n  \oplus
M_n)$ with $N$.   We sometimes simply write $(x_n)$ for the associated tuple in $N$.   

\begin{lemma} \label{lem1}  Let $A$ be any closed subalgebra of $N$ containing $A_0$.  Then $A$ is 
semisimple.  \end{lemma} \begin{proof}  For any nonzero $x = (x_n) \in A$, choose
$m$ and $i$ with $z = x_m e_i \neq 0$, where $(e_i)$ is the usual basis of 
$\Cdb^m$.   Choose $y_m \in M_m$ with $y_m z = e_i$, and otherwise set $y_n = 0$.
Then $y  = (y_n) \in A_0$, and the copy of $e_i$ is in the kernel of $I-yx$.  
Hence $I - yx$ is not invertible 
in $A^1$, and so $x$ is not in the Jacobson radical by a well known 
characterization of that radical.  Thus $A$ is
semisimple.  \end{proof} 

Endow $M_n$ with a norm $p_n(x) = \max \{ \Vert x \Vert , \Vert d_n x d_n^{-1} \Vert \}$.
Then $N \cong \oplus^\infty_n \, (M_n, p_n(\cdot))$ isometrically, and we 
write $p(\cdot)$ for the norm on the latter space, so 
$p((x_n)) = \sup_n \, p_n(x_n)$.   We sometimes view $p$ as the  norm on $N$ via 
the above identification.   
Let $L_n$ be the left shift on $\Cdb^n$, so that in particular $L_n e_1 = 0$.
Note that $d_n L_n d_n^{-1} = \frac{1}{c_n} \, L_n$, and that $p_n(L_n) = 1$ if $n \geq 2$.  
For $n, k \in \Ndb$ with $n \geq k$ define an `integer interval' $E_{n,k} = \Ndb_0 \cap [\frac{n}{k} ,
\frac{2n}{k}]$.  Set $\mu_{n,k} = |E_{n,k}|$ if $n \geq k$, with $\mu_{n,k} = 1$ if 
$n < k$.  Then $\mu_{n,k}$ is strictly positive for all $n, k$.  For $n \geq k$  
define $u_{n,k} = \frac{1}{\mu_{n,k}} \, \sum_{i \in E_{n,k}} \, (L_n)^i \in M_n$.
If $n < k$ set $u_{n,k} = I_n$.   Define $u_k = (u_{n,k})_{n \in \Ndb}$.  We have 
$$p_n(u_{n,k}) \leq \max_{i \in E_{n,k}} \, p((L_n)^i) \leq 1 , \qquad n \geq k,$$   
and so  $$p(u_k) \leq 1 , \qquad k \in \Ndb .$$

The operator algebra we are interested in is
$$A = \{ a \in N : p(a u_k -a) + p(u_k a - a) \to 0 \} .$$ 
This will turn out to be
the largest subalgebra of $N$ having $(u_k)$ as a cai. 
First,   a preliminary estimate:

\begin{lemma} \label{lem2}  Let $L \in {\rm Ball}(B)$ for a Banach algebra $B$. 
Suppose that  $E_1$ is a set of $\mu_1$ integers from $[0, n]$, and 
$E_2$ is a set of $\mu_2$ consecutive nonnegative integers.    If
$u_i = \frac{1}{\mu_{i}} \, \sum_{i \in E_{i}} \, L^i$ then
$$\Vert u_1 u_2 - u_2 \Vert \leq \frac{2 n}{\mu_2}.$$
 \end{lemma} \begin{proof}   If $n \geq \mu_2$ 
then 
$$\Vert u_1 u_2 - u_2 \Vert \leq \Vert u_1 \Vert \Vert u_2  \Vert + \Vert u_2  \Vert
\leq 2 \leq \frac{2n}{\mu_2}.$$
So we may assume that $n < \mu_2$.  Let $m_0 = \min \, E_2$.  
Then 
$$u_1 u_2 = \frac{1}{\mu_1 \, \mu_2} \, \sum_{j \in E_1, k \in E_2} \,
 L^{j+k}
= \sum_{m_0 \leq m < m_0 + n + \mu_2} \, \lambda_m L^m , $$
where $\lambda_m$ is $\frac{1}{\mu_1 \, \mu_2}$ times the number of pairs in 
$E_1 \times E_2$ which sum to $m$.  
  Since 
$$\mu_1 \leq n+ 1 \leq \mu_2,$$
and since the number of such pairs cannot exceed $\mu_1 = |E_1|$, 
we have $$0 \leq \lambda_m \leq \frac{1}{\mu_2}.$$  
If $m \in [m_0 + n, m_0 + \mu_2)$ then $m - k \in E_2$ 
for any integer $k$ in $[0,n]$, and so $m - E_1 \subset E_2$.
We deduce that $$\lambda_m = \frac{1}{\mu_2} , \qquad m \in [m_0 + n, m_0 + \mu_2
).$$ Since $u_2 = \frac{1}{\mu_2}  \, \sum_{m_0 \leq m < m_0 + \mu_2} \, L^m$ we have
$$u_1 u_2 - u_2 = \sum_{m_0 \leq m < m_0 + n} \,
(\lambda_m - \frac{1}{\mu_2}) \, L^m \, + \, \sum_{m_0 + \mu_2
\leq m <  m_0 + n + \mu_2} \, \lambda_m L^m .$$
No coefficient in the last sum has modulus greater than $\frac{1}{\mu_2}$, and there
are $2n$ nonzero coefficients, so 
$$\Vert u_1 u_2 - u_2 \Vert \leq \frac{2n}{\mu_2} \, \max_m \, 
\Vert L^m \Vert = \frac{2n}{\mu_2} .$$
as desired.   \end{proof}           

\begin{corollary} \label{isab}  Let $A = \{ a \in N : p(a u_k -a) + p(u_k a - a) \to 0 \}$.
Then   $A$ is a semisimple operator algebra with cai $(u_k)$, and $A_0$ 
is an ideal in $A$.
\end{corollary}
\begin{proof}   We first show $u_r \in A$ for all $r \in \Ndb$.  Let $k \geq r$.
If $n \geq k$ 
then $E_{n,k}$ is a subset of $[0, \frac{2n}{k}]$, and $\mu_{n,k}$ is 
either $\lfloor \frac{n}{k} \rfloor$ or $\lfloor \frac{n}{k}+1 \rfloor$.
By Lemma \ref{lem2}, we have 
$$p_n(u_{n,k} u_{n,r} - u_{n,r}) \leq \frac{2 \, \lfloor \frac{2n}{k} \rfloor}{\lfloor \frac{n}{r} \rfloor} , 
r \geq n \geq k .$$
If $n < k$ then $p_n(u_{n,k} u_{n,r} - u_{n,r}) = 0$.  If $k \geq 2 t r$ for an integer
$t > 1$ then $$\frac{2 \, \lfloor \frac{2n}{k} \rfloor}{\lfloor \frac{n}{r} \rfloor} \leq
\frac{\lfloor\frac{n}{tr} \rfloor}{\lfloor \frac{n}{r} \rfloor} \leq \frac{1}{t} .$$
Thus $p_n(u_{n,k} u_{n,r} - u_{n,r}) \leq \frac{2}{t}$ for $k \geq 2 t r$, so
$$p(u_{k} u_r - u_r) = \sup_n \, p_n(u_{n,k} u_{n,r} - u_{n,r}) \leq \frac{2}{t}
, \qquad k \geq 2 t r .$$
So $u_k u_r \to u_r$ with $k$, and so $u_r \in A$ for all $r  \in \Ndb$.      

It is now obvious that $A$, being a subalgebra of the operator algebra $N$,
 is an operator algebra with cai $(u_k)$.  
It is elementary that for any matrix $x$ in
the copy $M_n'$ of $M_n$ in $A_0$ we have $x u_k \to x$ and $u_k x \to x$,
since for example $u_k x = x$ for $k > n$. 
 Hence $A_0 \subset A$, so that  $A$ is semisimple by Lemma \ref{lem1}.
Since $M_n'$ is an ideal in $N$, so is $A_0$, giving the last statement.
  \end{proof}  

In the following result, and elsewhere, $\Vert \cdot \Vert$ denotes 
the usual norm on $M_n$ or on $\oplus^\infty \, M_n$.
\begin{lemma} \label{lem3}  For each $n \in \Ndb$ and $k \leq n$,
we have  $\Vert u_{n,k} \Vert \geq 1 - \frac{2}{k}$ and 
$\Vert u_{n,k}^3 \Vert \geq 1 - \frac{6}{k}$.
\end{lemma} \begin{proof}   
 If $\eta$ is the unit vector $(\frac{1}{\sqrt{n}}, \cdots , \frac{1}{\sqrt{n}})$ in $\Cdb^n$, then it is easy to see that
$$\langle (L_n)^k \eta, \eta 
\rangle = 1 - \frac{k}{n} , \qquad 0 \leq k \leq n. $$
Since $u_{n,k}$ is an average of powers $(L_n)^j$ with
$0 \leq j \leq \frac{2n}{k}$, we have
$$\langle u_{n,k} \eta, \eta
\rangle \geq 1 - \frac{\frac{2n}{k}}{n} = 1 - \frac{2}{k} .$$
Similarly, $u_{n,k}^3$ is a weighted average of powers $(L_n)^j$ with
$0 \leq j \leq \frac{6n}{k}$.
 \end{proof}

We note that the diagonal matrix units $e^n_{i,i}$ are projections, and are 
also minimal idempotents in $A$ (that is, 
have the property that 
$e A e = \Cdb e$).

\begin{theorem} \label{nowk}  $A$ is not weakly compact, and is not separable.
\end{theorem} \begin{proof}  Note that $A$ is an $\ell^\infty$-bimodule
via the action $$(\alpha_n) \cdot (T_n) =  (T_n)  \cdot (\alpha_n)
= (\alpha_n T_n) , \qquad (\alpha_n) \in \ell^\infty,
(T_n) \in A .$$
We will use this to embed $\ell^\infty$ isomorphically in $xAx$, where 
$x = u_{r}$ for large enough $r$.   Note that  
$$\ell^\infty \cdot x^3 = x (\ell^\infty \cdot x)x \subset xAx .$$
Choosing $r$ with
$1 - \frac{6n}{r} \geq \frac{1}{2}$, we have  that 
$\Vert u_{n,r}^3 \Vert \geq \frac{1}{2}$ for all $n \in \Ndb$
(recall $u_{n,r} = I$ if $n < r$).  Thus  for $\vec \alpha = (\alpha_n) \in \ell^\infty$ we have
 $$p(\vec \alpha \cdot x^3) \geq \Vert \vec \alpha \cdot x^3 \Vert =  \Vert \vec \alpha \cdot u_{r}^3 \Vert = \sup_n \, |\alpha_n| \Vert u_{n,r}^3 \Vert \geq \frac{1}{2} \sup_n \, |\alpha_n|,$$
and so the map $\vec \alpha  \mapsto \vec \alpha \cdot x^3$ is a bicontinuous injection of
$\ell^\infty$  into $xAx$.  Thus  $A$ is not weakly compact, nor separable. 
\end{proof}

\begin{lemma} \label{lem4}  If $T = (T_n) \in A$, then $\Vert d_n T_n d_n^{-1} \Vert \to 0$ as $n \to \infty$.
Thus the spectral radius $r(T_n) \to 0$  as $n \to \infty$.
\end{lemma} \begin{proof} 
Given $\epsilon > 0$ there exists an  $m \in \Ndb$ such that 
$$p_n(u_{n,m} T_n - T_n) + p_n(T _n u_{n,m}  - T_n)  < \frac{\epsilon}{2} p(u_m T - T) + p(T u_m  - T)  < \frac{\epsilon}{2} ,
n \in \Ndb. $$ We have noted that $d_n L_n d_n^{-1} = \frac{1}{c_n} L_n,$ and for $n \geq m$ the operator 
$u_{n,m}$ is an average of powers $L^j_n$, so for $n \geq m$ we have 
$$\Vert d_n u_{n,m} d_n^{-1} \Vert \leq \max_{j \in \Ndb} \, \Vert d_n L_n^j  d_n^{-1} \Vert \leq  \frac{1}{c_n} .$$
Thus $$\Vert d_n T_n u_{n,m}  d_n^{-1} \Vert \leq \frac{1}{c_n}  \Vert d_n T_n d_n^{-1} \Vert \leq \frac{1}{c_n}  p(T) .$$
Consequently,  for $n \geq m$ the  quantity $\Vert d_n T_n d_n^{-1} \Vert$ is dominated by 
$$\Vert d_n (T_n u_{n,m} - T_n) d_n^{-1} \Vert + \Vert d_n T_n u_{n,m} d_n^{-1} \Vert \leq 
p_n(T_n u_{n,m} - T_n)  + \frac{1}{c_n}  p(T) \leq \frac{\epsilon}{2}  + \frac{1}{c_n}  p(T).$$
The result is clear from this.  
\end{proof}   

For a matrix $B$ write $\overline{\Delta}_U B$ for the upper triangular projection of $B$ (that is, we change $b_{ij}$ to $0$ if
$i > j$).  Similarly, write $\Delta_L B$ for the strictly lower triangular part of $B$.
In the next results, as usual ${r \choose s} = 0$ if $0 \leq r < s$ are integers.  

\begin{lemma} \label{lem5}  If $0 \neq T = (T_n) \in A$, and $\epsilon > 0$ is given, there exist 
$k, m \in \Ndb$ such that for all $r \in \Ndb_0$ and $n \geq \max \{ k, m \}$, we have 
$$\Vert (\overline{\Delta}_U T_n)^r \Vert \leq \sum_{s=0}^{k-1} \, {r \choose s} \, (2 p(T))^r \, \epsilon^{r-s} .$$
\end{lemma} \begin{proof} 
The $i$-$j$ entry $T_{n,i,j}$ of $T_n$ equals
$\langle T_n e_j , e_i \rangle = c_n^{j-i} \, \langle  d_n T_n d_n^{-1} e_j , e_i \rangle$, and so 
$$|T_{n,i,j} | = c_n^{j-i} \, |\langle  d_n T_n d_n^{-1} e_j , e_i \rangle|  \leq c_n^{j-i} \, p_n(T_n) , \qquad  T = (T_n) \in A.$$  
It follows from this that
$$\Vert \sum_{j = 1}^{n-r} \, T_{n,j+r,j} \, E_{j+r,j} \Vert = \max_{j \leq n-r} \, |T_{n,j+r,j} | \leq c_n^{-r} \, p_n(T_n) ,$$ if $r < n$.
Since $\sum_{r =1}^{n-1} \, (\sum_{j = 1}^{n-r} \, T_{n,j+r,j} \, E_{j+r,j}) = \Delta_L T_n$, we deduce that
\begin{equation} \label{lse} \Vert \Delta_L T_n \Vert  = \Vert T_n -  \overline{\Delta}_U T_n 
\Vert \leq \sum_{r =1}^{n-1} \, c_n^{-r} \, p_n(T_n) \leq \frac{p_n(T_n)}{c_n - 1} \leq
\frac{p(T)}{c_n - 1}. \end{equation}  

Given $\epsilon > 0$ choose $k$ with $p(u_k T - T) < \epsilon p(T),$  and let $n \geq k$.
 Then $$\Vert u_{n,k} T_n - T_n \Vert \leq p_n(u_{n,k} T_n - T_n) < \epsilon p(T),$$ and so
$$\Vert u_{n,k} \overline{\Delta}_U T_n - \overline{\Delta}_U T_n \Vert  \leq \epsilon p(T) + \Vert (u_{n,k} - I) (T_n -  \overline{\Delta}_U T_n) \Vert \leq p(T) (\epsilon + \frac{2}{c_n - 1}),$$
since $$u_{n,k} \overline{\Delta}_U T_n - \overline{\Delta}_U T_n = (I - u_{n,k}) (T_n -  \overline{\Delta}_U T_n)  + (u_{n,k} T_n - T_n) .$$
Let $S_1 = u_{n,k} \overline{\Delta}_U T_n$ and $S_2 = \overline{\Delta}_U T_n - S_1$, then 
$\Vert S_2 \Vert \leq p(T) (\epsilon + \frac{2}{c_n - 1}),$ by the last displayed equation.   Also,
$$\Vert S_1 \Vert \leq \Vert \overline{\Delta}_U T_n \Vert \leq p(T) + \Vert (I - \overline{\Delta}_U) T_n \Vert \leq p(T) + \frac{p(T)}{c_n - 1}
= p(T) \frac{c_n}{c_n - 1}$$
by (\ref{lse}).  

Now $\overline{\Delta}_U T_n = S_1 + S_2$,
so $(\overline{\Delta}_U T_n)^r$ is a sum from $s = 0$ to $r$, of ${r \choose s}$ times terms which are a product of $r$ factors, 
$s$ of which are $S_1$ and $r-s$ of which are $S_2$.    Note that any product of upper triangular $n \times n$ matrices that has $k$ or more 
factors which equal $S_1$, is zero.  This is because multiplication of an upper triangular matrix $U$ by $u_{n,k}$ (and hence by $S_1$) 
decreases the 
number of nonzero `superdiagonals' of $B$ by a number   
$\geq \frac{n}{k}$, so after $k$ such multiplications we are left with the zero matrix.   Thus we can assume that $s < k$ above.
Using the estimates at the end of the last paragraph, we deduce that 
$$\Vert (S_1 + S_2)^r \Vert \leq \sum_{s=0}^{k-1} \, {r \choose s} \Vert S_1 \Vert^s \Vert S_2 \Vert^{r-s} \leq
\sum_{s=0}^{k-1} \, {r \choose s} \, (p(T) \frac{c_n}{c_n - 1})^s \, (p(T) (\epsilon + \frac{2}{c_n - 1}))^{r-s}. $$
Since $c_n \to \infty$ we may choose $m$ such that $\frac{c_n}{c_n - 1} < 2$ and $\epsilon + \frac{2}{c_n - 1} < 2 \epsilon$ for 
all $n \geq m$.  Thus for $n \geq \max \{ k, m \}$, we have
$$\Vert (\overline{\Delta}_U T_n)^r  \Vert = \Vert (S_1 + S_2)^r \Vert \leq \sum_{s=0}^{k-1} \, {r \choose s} \,  (2 p(T))^r \epsilon^{r-s}$$
as desired.  
\end{proof}   

For $k \in \Ndb$ and
positive numbers $b, \epsilon$,
define a quantity  $K(k,b,\epsilon) = \frac{1}{2 b 
(1- \epsilon) \,  \epsilon^k}$.  

\begin{lemma} \label{lem6}  If $0 \neq T = (T_n) \in A$, and $\epsilon > 0$ is given, there exist 
$k, m \in \Ndb$ such that for all $\lambda \in \Cdb$ with $|\lambda| > 4 p(T) \epsilon$, 
and $n \geq \max \{ k, m \}$, we have $\lambda I - \overline{\Delta}_U T_n$ and $\lambda I - T_n$ invertible in $M_n$, and both 
$$\Vert (\lambda I - \overline{\Delta}_U T_n)^{-1} \Vert \leq K(k,
p(T), \epsilon)$$ and $$\Vert  (\lambda I - T_n)^{-1} \Vert \leq 2 K(k,
p(T), \epsilon).$$
\end{lemma} \begin{proof} If  $|\lambda| > 2 p(T) \epsilon$ then  $$\sum_{r=0}^{\infty} \, \Vert \lambda^{-r-1} \,  (\overline{\Delta}_U T_n)^r
\Vert \leq |\lambda|^{-1} \,  \sum_{r=0}^{\infty} \,  \sum_{s=0}^{k-1} \, {r \choose s} \, (\frac{2 p(T)}{|\lambda|})^r \, \epsilon^{r-s} ,$$
by Lemma \ref{lem5}, for $n \geq \max \{ k, m \}$, where $k, m$ are as in that lemma.   However the latter quantity equals
$$|\lambda|^{-1} \,  \sum_{s=0}^{k-1} \, \sum_{r=0}^{\infty} \,   {r \choose s} \, (\frac{2 p(T) \epsilon}{|\lambda|})^{r-s} \, 
(\frac{2 p(T)}{|\lambda|})^{s} = |\lambda|^{-1} \,  \sum_{s=0}^{k-1} \, 
(\frac{2 p(T)}{|\lambda|})^{s}  \, (1 - \frac{2 p(T) \epsilon}{|\lambda|})^{-s-1}$$
using the binomial formula.   This is finite, so $\sum_{r=0}^{\infty} \,  \lambda^{-r-1} \,  (\overline{\Delta}_U T_n)^r$ converges, 
and this  is clearly an inverse for  $\lambda I - \overline{\Delta}_U T_n$.  
If $|\lambda| > 4 p(T) \epsilon$, then the sum in the 
last displayed equation is dominated by
$$\frac{1}{4 p(T) \epsilon}  \,  \sum_{s=0}^{k-1} \,
(\frac{1}{2  \epsilon})^s \, 2^{s+1} = 
\frac{1}{2 p(T) (1 -  \epsilon)} \frac{1 - \epsilon^k}{\epsilon^k}
\leq K(k,p(T),\epsilon) .$$
 We also obtain 
\begin{equation} \label{deleq} \Vert (\lambda I - \overline{\Delta}_U T_n)^{-1} \Vert
\leq K(k,p(T), \epsilon). \end{equation} 

By increasing $m$ if necessary, we can assume that $c_n - 1 > 2 \, p(T) \, K(k,
p(T), \epsilon)$.   Then by (\ref{lse}) we have 
$$\Vert T_n - \overline{\Delta}_U T_n \Vert \leq \frac{p(T)}{c_n - 1} < \frac{1}{2 K(k,
p(T), \epsilon)}.$$ 
A simple consequence of the Neumann lemma is that if $R$ is invertible and $\Vert H \Vert < \frac{1}{2  \Vert R^{-1} \Vert}$, 
then $R + H$ is invertible and $\Vert (R + H)^{-1} \Vert \leq 2 \Vert R^{-1} \Vert$.  Setting $R = \lambda I - \overline{\Delta}_U T_n$ and
$H =  \overline{\Delta}_U T_n - T_n$, we have $$\Vert H \Vert < \frac{1}{2 K(k,
p(T), \epsilon)} < \frac{1}{2  \Vert R^{-1} \Vert}$$  by (\ref{deleq}).   Hence $R + H = \lambda I - T_n$ is
invertible, and by (\ref{deleq}) again the norm of its inverse is dominated by $2 \Vert R^{-1} \Vert \leq 2
K(k, p(T), \epsilon).$
 \end{proof}   

The quantity $K(k,p(T), \epsilon)$ above is independent of $n$, which gives:

\begin{corollary} \label{isnotma}     The spectrum of every element of $A$  is finite 
or a null sequence and zero.   Hence $A$ is a modular annihilator algebra.
\end{corollary}
\begin{proof}   Let $0 \neq T = (T_n) \in A$.  We will show that the spectrum of $T$ is finite 
or a null sequence and zero.   It is sufficient to show that if $\epsilon > 0$ is given, there exists $m_0 \in \Ndb$
such that if $|\lambda| > 4 p(T) \epsilon$, and if $\lambda$ is not in the spectrum of $T_1, \cdots , T_{m_0}$,
then $\lambda \notin {\rm Sp}_A(T)$.   So assume these conditions, and let $m_0 = \max \{ k, m \}$ as in 
Lemma \ref{lem6}.  For $n \geq m_0$ we have by Lemma \ref{lem6} that $\lambda I - T_n$ is invertible, and the usual matrix norm of its inverse
is bounded independently of $n$.   
 By assumption this is also true for $n < m_0$.   By Lemma \ref{lem4} there is a $q$ such that
$\Vert d_n T_n d_n^{-1} \Vert < \epsilon$ for $n \geq q$.  If $|\lambda| > \epsilon$ then $(\lambda I - T_n)^{-1} = 
\sum_{r=0}^{\infty} \,  \lambda^{-r-1} \,  T_n^r$ and 
$$\Vert d_n (\lambda I - T_n)^{-1} d_n^{-1} \Vert = 
\Vert \sum_{r=0}^{\infty} \,  \lambda^{-r-1} \,  d_n T_n^r d_n^{-1} \Vert \leq \sum_{r=0}^{\infty} \, |\lambda|^{-r-1} \, \epsilon^r
= |\lambda|^{-1} \, (1 - \frac{\epsilon}{|\lambda|})^{-1} .$$
Thus $(p_n((\lambda I - T_n)^{-1}))$ is bounded independently of $n$.   Hence  $((\lambda I - T_n)^{-1}) \in N$,
and this is an inverse in $N$ for $\lambda I - T$.  Thus the spectrum of $T$ in $N$ is finite 
or a null sequence and zero.    The  spectrum in $A$ might be bigger, but since the boundary of its spectrum
cannot increase, Sp$_A(T)$ is also finite 
or a null sequence and zero.

The last statements follow from \cite[Chapter 8]{Pal}.  \end{proof}

We point out some more  features of our example $A$, in hope that these may further its future use as a counterexample in
the subject:

\begin{proposition} \label{multis}  The multiplier algebra of $A$ may be taken to be
$\{ x \in N : x A + Ax \subset A \}$.  This is also valid with $N$ replaced by $M$. 
\end{proposition}

\begin{proof}
Viewing $M = \oplus_n^\infty \, (M_n  \oplus
M_n)$ as represented on $H =\oplus^2_n (\Cdb^n \oplus \Cdb^n)$, it is clear that $D_0$, and hence also $A$, 
acts nondegenerately on $H$.  So the multiplier algebra $M(A)$ may be viewed as a subalgebra
of $B(H)$.
We also see that the weak* continuous extension $\tilde{\pi} : A^{**} \to N$ of the `identity map' on $A$,
is a completely isometric homomorphism from the copy of $M(A)$ in $A^{**}$  onto   the copy of $M(A)$ in $B(H)$,
and in particular, the latter is contained in $N$.    So the latter is
$M(A) = \{ x \in N : x A + Ax \subset A \}$.    A similar argument works with $M$ replaced by $N$.
\end{proof}

We note that if $D_n$ is the commutative diagonal $C^*$-algebra
in $M_n$, then there is a natural isometric copy $D$  of $\oplus_n^\infty \, D_n$ inside $N$, namely 
the tuples $((x_n , x_n))$ for a bounded sequence  $x_n \in D_n$.   

We assume henceforth that $c_n > 1$ for all $n$.

\begin{proposition} \label{diaga}  The diagonal $\Delta(A) = A \cap A^*$ of $A$ equals
 the natural copy $D_0$ of   the 
$c_0$-sum $C^*$-algebra $\oplus_n^\circ \, D_n$ inside $A$.  
\end{proposition}
\begin{proof}       If $((x_n , d_n x_n d_n^{-1}))$ is selfadjoint,
then $x_n$ is  selfadjoint, and $d_n x_n d_n^{-1}$ is selfadjoint, which forces $d_n^2$ to commute with $x_n$.
However this implies that $x_n$ is diagonal.  Since $\Delta(N) = N \cap N^*$ is spanned by its selfadjoint 
elements it follows that $\Delta(N) = D$.
Therefore  $\Delta(A) = D \cap A$,  and this contains $D_0$ since $D_0 \subset A_0 \subset A$
by Corollary \ref{isab}.  The reverse containment follows easily from Lemma \ref{lem4}, but we give
a shorter proof. Let  $(a_n) \in  D \cap A$, with $a_n \in D_n$
for each $n$.  If $\epsilon > 0$ is given, choose $k$ such that $p(u_k (a_n)  - (a_n)) < \epsilon$.
Choose $m$ with  $u_{n,k}$ strictly upper triangular for all $n \geq m$.  Then for 
$n \geq m$ we have $|a_n(i) |$, which is the modulus of the  $i$-$i$ entry of $(u_k (a_n)  - (a_n))$, is
dominated by $$\Vert  u_{n,k} \, a_n - a_n \Vert \leq p(u_k  \, (a_n)  - (a_n)) < \epsilon.$$
Thus $\Vert a_n \Vert < \epsilon$ for $n \geq m$, so that $(a_n) \in D_0$.     \end{proof}

\begin{corollary} \label{ispro}    Projections in  $A^{**}$ which are both open and closed,
or equivalently which are in $M(A)$, must be also in $D$.  Thus they are diagonal matrices
with  $1$'s as the only permissible nonzero entries.  
\end{corollary}
\begin{proof}   This follows from Proposition \ref{multis} and the fact from the proof of 
Proposition \ref{diaga} that $\Delta(N) = D$.  
\end{proof}

Note that the natural approximate identity for $\Delta(A) = D_0$ is not an approximate identity for $A$
(since $D_0 A \subset A_0 A \subset A_0 \neq A$).   Thus $A$ is not $\Delta$-dual in the sense of \cite{ABS}.
In \cite{ABR} we showed that any operator algebra which  is weakly compact is nc-discrete, and we asked if every semisimple 
modular annihilator algebra was nc-discrete.
 To see that our example  $A$ is not nc-discrete  in the sense of \cite{ABS} note that $A_0$ is an r-ideal in $A$
(and an $\ell$-ideal), and its support projection $p$  in $A^{**}$, which is central in $A^{**}$, coincides with the support projection of $D_0$ in $A^{**}$, 
and this is an open projection in  $A^{**}$ which we will show is not closed.

\begin{corollary} \label{iscr}    The algebra $A$ above is not nc-discrete.  
\end{corollary}
\begin{proof}  We saw that $p$ above was open.   If $p$ also was closed in $A^{**}$, or equivalently in $M(A)$, then 
$\tilde{\pi}(1-p)$ would be a nonzero central projection in the  copy of $M(A)$ in $M$.
Also $\tilde{\pi}(1-p) e^n_{i,i}$ is nonzero for some $n$ and $i$, because the SOT sum of the $e^n_{i,i}$ in $M$ is $1$.  
On the other hand, since $e^n_{i,i}$ is in the ideal supported by $p$ we have $$\tilde{\pi}(p) e^n_{i,i}  = 
 \tilde{\pi}(p \, e^n_{i,i}) = \tilde{\pi}(e^n_{i,i}) =e^n_{i,i},$$ and so 
$$\tilde{\pi}(1-p) \,  e^n_{i,i} = \tilde{\pi}(1-p) \,  \tilde{\pi}(p) \, e^n_{i,i}  = 0 .$$  This contradiction shows that
 $A$ is not nc-discrete.    \end{proof}  

Indeed $A_0$ is a nice r- and $\ell$-ideal in $A$ which is supported by an  open projection which  
is not one of the obvious projections, and is not any projection in $M(A)$.
Note that $A$  
is not 
a left or right annihilator algebra in the sense of e.g.\ \cite[Chapter 8]{Pal},
since for example by \cite[Chapter 8]{Pal} this implies that $A$ is compact, whereas 
above we showed that $A$ is not even weakly compact.    The spectrum of $A$ 
is discrete, and every left ideal of $A$ contains a minimal 
left ideal, by \cite[Theorem 8.4.5 (h)]{Pal}.  Also every idempotent in $A$ belongs
to the socle by \cite[Theorem 8.6.6]{Pal}, hence to 
$A_{00}$ by the next result.   From this it is clear what all the 
idempotents  in $A$ are.     

\begin{corollary} \label{issoc} The maximal modular right (resp.\ left) ideals in 
$A$ are exactly the ideals of the form $(1-e) A$ (resp.\ $A(1-e)$) for 
a minimal idempotent $e$ in $A$ which is the canonical copy in $A$ of
a minimal idempotent in $M_n$ for some $n \in \Ndb$.  The socle  of $A$
is $A_{00}$, namely the set of $(a_n) \in A$ with $a_n = 0$ except for at most finitely 
many $n$.    
\end{corollary}
\begin{proof}   Let $e = (e_n)$ be a (nonzero) minimal idempotent in $A$. 
Then $e_n$ is an idempotent in $M_n$ for each $n$.
If $e^n_{i,i}$ is as above, then because the SOT sum of the $e^n_{i,i}$ in $M$ is $1$,
we must have $e e^n_{i,i} e \neq 0$ for some $n$ and $i$.
Since $e$ is minimal, for such  $n$, $e$ is in the copy of $M_n$ in $A_0$.
So this $n$ is unique, and $e$ is clearly a 
minimal idempotent in this copy of $M_n$ in $A_0$.  Now it is 
easy to see the assertion about the socle of $A$.   By \cite[Proposition 8.4.3]{Pal},
it follows that the maximal modular left ideals  in $A$  are the  ideals
$A(1-e)$ for an $e$ as above.
We have also used the fact here that $A$ has no right annihilators in $A$.  
Similarly for right ideals. 
 \end{proof}

\begin{corollary} \label{oncom}    The only compact projections in $A^{**}$ for 
the algebra $A$ above are the obvious `main diagonal' ones; that is the projections in 
$D_0 \cap A_{00}$.  
\end{corollary}
\begin{proof}  Let $T = (T_n) \in A$, and $\epsilon \in (0,\frac{1}{4 p(T)})$ be given.  As in the proof
of Corollary \ref{isnotma} there exists $m_0 \in \Ndb$
such that if $|\lambda| > 4 p(T) \epsilon$
then
 $\lambda I - T_n$ is invertible for $n \geq m_0$, and the usual matrix norm of its inverse
is bounded independently of $n  \geq m_0$.   As in that proof, if $S_n = T_n$ for $n \geq m_0$, and 
$S_n = 0$ for $n < m_0$, then $\lambda I - S$ is invertible in $N$.  Thus the spectral 
radius $r(S) \leq 4 p(T) \epsilon < 1$. Hence $\lim_{k \to \infty} \,  S^k = 0$ in norm.     
Let $q$ be the central projection in $A$ corresponding 
to the identity of $\oplus_{n=1}^{m_0 -1} \, M_n$.
If now also $T \in \frac{1}{2} {\mathfrak F}_A$, then $T^k$ converges 
weak* to its peak projection $u(T)$ weak* 
by \cite[Lemma  3.1, Corollary 3.3]{BNII}, as $k \to \infty$.
Thus $T^k q \to u(T) q$ and $T^k (1-q) = S^k  \to u(T) (1-q)$ weak*.
Clearly it follows that $u(T) q$ is a projection in $A$, hence in $D_0 \cap A_{00}$ as we 
said above. On the other hand, since $S^k  \to 0$ we have $u(T) (1-q) = 0$.
Thus $u(T)$ is a projection in $D_0 \cap A_{00}$.

Finally we recall from \cite{BNII} that the compact projections in $A^{**}$ 
are decreasing limits of such $u(T)$.    Thus any 
compact projection is in $D_0 \cap A_{00}$. \end{proof}   

One may ask if there exists a  {\em commutative} semisimple approximately unital operator algebra which is a modular annihilator algebra but is not weakly compact.   After this paper was submitted 
we were able to check that the algebra
constructed in  \cite{BRIV} was such an algebra.  However this example
 is quite a bit more complicated than the interesting noncommutative example above.   

\medskip

We end this paper by mentioning 
 a complement to the example above.   In \cite[p.\ 76]{ABR} we asked if for an approximately unital commutative operator algebra $A$, which is an ideal in its bidual (which means that multiplication by any fixed element of $A$ is weakly compact), is the spectrum of every element at most countable; and is the spectrum of $A$ scattered?   In particular, is it a 
modular annihilator algebra  (we recall that 
compact semisimple algebras are modular
annihilator algebras \cite[Chapter 8]{Pal}).
There is in fact a simple counterexample to these questions, which is quite well known in other contexts. 
The example may be described either in the operator theory language of weighted unilateral shifts, and the $H^p(\beta)$ spaces that occur there, or in the Banach algebra language of weighted convolution algebras $l^p(\Ndb_0,\beta)$.  
These are equivalent (in particular, $H^2(\beta) = l^2(\Ndb_0,\beta)$).  For brevity we just mention the
 operator theory angle (a previous draft had a much fuller exposition).
Let $T$ be a weighted unilateral shift which is  one-to-one (that is, 
none of the weights are zero), and let $A$ be  the algebra  generated by $T$.
Then $A$ is  isomorphic to a Hilbert space if  $T$ is  {\em strictly cyclic} in Lambert's sense \cite{Sh}.
Central to the theory of weighted shifts is the convolution algebra
$H^2(\beta) =  l^2(\Ndb_0,\beta)$, and its space of `multipliers' $H^\infty(\beta)$ .  These spaces can canonically be viewed as spaces of converging (hence analytic) power series on a disk, via the map $(\alpha_n) \mapsto \sum_{n = 0}^\infty \, \alpha_n z^n$.  
By the well known theory in \cite{Sh}, $T$ is 
unitarily equivalent to multiplication by $z$ on $H^2(\beta)$, the latter viewed as a space of power series on the disk 
of radius $r(T)$.   In our strictly cyclic case, $A$, which equals its weak closure, is unitarily equivalent via the same unitary to  
 $H^\infty(\beta)$.  Since $H^\infty(\beta) H^2(\beta) \subset H^2(\beta)$ and the constant polynomial is in $H^2(\beta)$, it is clear that $H^\infty(\beta) \subset H^2(\beta)$.  However, since the constant polynomial $1$ is a strictly cyclic vector, we in fact have $H^\infty(\beta) = H^\infty(\beta) 1 = H^2(\beta)$  (see p.\ 94 in that reference).
On the same page of that reference we see that the closed disk $D$ of radius $r(T)$ is the maximal ideal space of $H^\infty(\beta)$, and the spectrum of any $f \in H^\infty(\beta)$ is $f(D)$.  In particular, Sp$_A(T) =D$, and $A$ is semisimple.

\medskip

{\em Acknowledgements.}  We thank Garth Dales and Tomek Kania for discussions on 
algebraically finitely generated ideals in 
Banach algebras, and in particular drawing our attention to the results in \cite{ST}.  The first author wishes to thank the departments at the Universities of 
Leeds and Lancaster, and in particular the second author and Garth Dales,
for their warm hospitality during a visit in April--May 2013.     We also gratefully acknowledge   support from UK research council
grant  EP/K019546/1
for largely funding that visit.  We thank Ilya Spitkovsky 
for some clarifications regarding roots of operators and results in \cite{MP}, and Alex Bearden for spotting several typos
in the manuscript.


\begin{thebibliography}{99} 

\bibitem{Ake2} C. A. Akemann, {\em Left ideal structure of
    $C^*$-algebras}, J.\ Funct.\ Anal.\ \textbf{6} (1970), 305--317.  

\bibitem{SOC}  W. B. Arveson, {\em Subalgebras of $C^{*}$-algebras,}
 Acta Math.\  {\bf 123} (1969), 141--224.  

 \bibitem{AE2}   
L. Asimow and  A. J. Ellis,  {\em On Hermitian functionals on unital Banach algebras,}
 Bull. London Math. Soc. {\bf 4} (1972), 333--336. 

 

 \bibitem{AE}   
L. Asimow and  A. J. Ellis, {\em Convexity theory and its applications in functional analysis,}
London Mathematical Society Monographs, 16, Academic Press
London-New York, 1980.  

 \bibitem{ABS}  
M. Almus, D. P. Blecher, and S. Sharma,  {\em Ideals and structure of operator algebras,}   J.\
Operator Theory {\bf 67} (2012), 397--436.  

\bibitem{BBS}  C. A. Bearden, D. P. Blecher and S. Sharma, {\em On positivity and roots in operator algebras,} Preprint, 2013.   


\bibitem{Bnew}  D.\  P.\ Blecher, {\em Noncommutative peak interpolation
revisited}, to appear Bull.\ London Math.\ Soc.
 
\bibitem{BHN}  D. P. Blecher, D. M. Hay, and
M. Neal, {\em Hereditary subalgebras of operator algebras,} J.\
Operator Theory {\bf 59} (2008), 333-357.

\bibitem{BLM}  D. P. Blecher
and C.  Le Merdy, {\em Operator algebras and their modules---an
operator space approach,} Oxford Univ.\  Press, Oxford (2004).  


\bibitem{BNI} D. P. Blecher
and M. Neal, {\em Open projections in operator algebras I: Comparison Theory},
Studia Math. {\bf 208} (2012), 117-150.



\bibitem{BNII} D. P. Blecher
and M. Neal, {\em Open projections in operator algebras II: compact projections,} Studia Math.\ {\bf 209} (2012), 203-224.  

 \bibitem{ABR}  
M. Almus, D. P. Blecher and C. J. Read, {\em Ideals and hereditary subalgebras  in operator algebras},
Studia Math.\  {\bf 212} (2012), 65--93.  

\bibitem{BRI}  D. P. Blecher and C. J. Read, {\em  Operator algebras with contractive approximate identities,}
J. Functional Analysis {\bf 261} (2011), 188-217.  

\bibitem{BRII}  D. P. Blecher and C. J. Read, {\em  Operator algebras with contractive approximate identities II,} J. Functional Analysis {\bf 261} (2011), 188-217.
 
\bibitem{BRIV}  D. P. Blecher and C. J. Read, {\em  Operator algebras with contractive approximate identities IV: a large operator 
algebra in $c_0$,}  ArXiV Preprint (2013).

\bibitem{CGK}  B. Cascales,  A. J. Guirao, and V. A  Kadets, {\em Bishop-Phelps-Bollobás type theorem for uniform algebras,}
Adv. Math. {\bf 240} (2013), 370--382.

 \bibitem{Dal}  H. G. Dales, {\em Banach algebras and automatic continuity},
London Mathematical Society Monographs.
New Series, 24, Oxford Science Publications.
The Clarendon Press, Oxford University Press, New York, 2000.  

\bibitem{Haase} M. Haase, {\em The functional calculus for sectorial operators,}
 Operator Theory: Advances and Applications, 169, Birkhäuser Verlag, Basel, 2006.  

\bibitem{Hay}  D. M. Hay, {\em Closed projections and peak interpolation for operator algebras,}
  Integral Equations Operator Theory  {\bf 57}  (2007),  491--512.  

\bibitem{LRS}  C-K. Li, L. Rodman, and I. M. Spitkovsky, {\em
 On numerical ranges and roots,}  J.
Math. Anal. Appl. 282 (2003),  329--340.

\bibitem{KR}  
R. V. Kadison and J. R. Ringrose,
{\em Fundamentals of the theory of operator algebras,} Vol. 1, 
Graduate Studies in Mathematics,
Amer.\ Math.\ Soc.\ Providence, RI, 1997.  

\bibitem{LN}  K. B. Laursen and M. M.  Neumann, {\em An introduction to local spectral theory,}
 London Mathematical Society Monographs, New Series, 20, The Clarendon Press, Oxford University Press, New York, 2000.

\bibitem{Li}   B. R. Li,
{\em Real operator algebras}, World Scientific Publishing Co., Inc., River Edge, NJ, 2003.  

\bibitem{MP}  V. I Macaev and Ju. A. Palant,  {\em On the
powers of a bounded dissipative operator} (Russian),
  Ukrain. Mat. Z.  {\bf 14}  (1962), 329--337.

\bibitem{Moore}   R. T. Moore, {\em Hermitian functionals on B-algebras and duality characterizations of $C^∗$-algebras,}  Trans. Amer. Math. Soc. {\bf 162} (1971),  253--265.

 \bibitem{Pal} T. W. Palmer, {\em Banach algebras and the general
theory of $*$-algebras, Vol.\ I.\ Algebras and Banach algebras,}
Encyclopedia of Math.\ and its Appl., 49, Cambridge University
Press, Cambridge, 1994.  

\bibitem{Ped} G. K. Pedersen, {\em $C^*$-algebras and their automorphism
groups,} Academic Press, London (1979).

\bibitem{Read}  C. J. Read,   {\em On the quest for positivity in operator algebras,}
 J. Math. Analysis and Applns.\  {\bf 381} (2011), 202--214.  

\bibitem{Sh}  A. L. Shields, Weighted shift operators and analytic function theory, In {\em Topics in operator theory,}  pp.\ 49--128. Math. Surveys, No. 13, Amer. Math. Soc., Providence, R.I., 1974.  



\bibitem{ST}  A. M. Sinclair and A. W.  Tullo, {\em  
Noetherian Banach algebras are finite dimensional},
Math. Ann. {\bf 211} (1974), 151--153. 

\bibitem{NF}   B. Sz.-Nagy, C. Foias, H.  Bercovici, and L.  Kerchy, {\em  Harmonic analysis of operators on Hilbert space,} Second edition,  Universitext. Springer, New York, 2010. 
\end{thebibliography}
\end{document}